\documentclass[12pt,a4paper, reqno]{amsart}

\usepackage[a4paper, total={6.5in, 9.5in}]{geometry}
\usepackage{graphicx}%
\usepackage{multirow}%
\usepackage{amsmath,amssymb,amsfonts}%
\usepackage{amsthm}%
\usepackage{mathrsfs}%
\usepackage[title]{appendix}%
\usepackage{xcolor}%
\usepackage{textcomp}%
\usepackage{manyfoot}%
\usepackage{booktabs}%
\usepackage{algorithm}%
\usepackage{algorithmicx}%
\usepackage{algpseudocode}%
\usepackage{listings}%
\usepackage{nicematrix}% matrices with labels around border
\usepackage{enumitem}

\usepackage{pifont}

\usepackage{graphicx}
\usepackage{hyperref}
\hypersetup{
    hidelinks
}
\usepackage{float}
\usepackage{mathtools}
\usepackage{cleveref}

\newcommand{\NN}{\mathbb N}
\newcommand{\RR}{\mathbb R}
\newcommand{\QQ}{\mathbb Q}
\newcommand{\ZZ}{\mathbb Z}

\newcommand{\MC}{\mathcal C}
\newcommand{\MD}{\mathcal D}
\newcommand{\MF}{\mathcal F}

\newcommand{\ML}{\mathcal L}
\newcommand{\MM}{\mathcal M}
\newcommand{\MT}{\mathcal T}

\DeclareMathOperator{\cone}{Cone}
\DeclareMathOperator{\supp}{Supp}
\DeclareMathOperator{\sign}{sgn}

\DeclareMathOperator{\nullity}{nullity}
\DeclareMathOperator{\lcm}{lcm}

\theoremstyle{plain}
\newtheorem{theorem}{Theorem}[section]
\newtheorem*{theorem*}{Theorem}
\newtheorem{lemma}[theorem]{Lemma}
\newtheorem{proposition}[theorem]{Proposition}
\newtheorem*{proposition*}{Proposition}
\newtheorem{corollary}[theorem]{Corollary}
\newtheorem{conjecture}[theorem]{Conjecture}

\theoremstyle{definition}
\newtheorem{definition}[theorem]{Definition}
\newtheorem*{definition*}{Definition}
\newtheorem{example}[theorem]{Example}

\theoremstyle{remark}
\newtheorem{remark}[theorem]{Remark}
\newtheorem{question}[theorem]{Question}

\allowdisplaybreaks

\title{Distance Reducing Markov Bases}
\author{Oliver Clarke and Dimitra Kosta}

\address{Dimitra Kosta, School of Mathematics, University of Edinburgh and Maxwell Institute for Mathematical Sciences, United Kingdom }
\email{D.Kosta@ed.ac.uk}

\address{Oliver Clarke, School of Mathematics, University of Edinburgh and Maxwell Institute for Mathematical Sciences, United Kingdom }
\email{oliver.clarke@ed.ac.uk}

\date{}

\begin{document}

\begin{abstract}
    The distance reducing property for Markov bases is an important property that provides a bound on the mixing time of the associated Markov chain. The goal of this project is to understand properties of distance-reducing Markov bases. We explore the distance reducing property for monomial curves and give a complete characterisation of distance reduction in the case of complete intersection monomial curves. Our characterisation carefully uses the notion of gluings for numerical semigroups. We also characterise the distance reducing property for non-complete intersection monomial curves in small dimensions. We also explore the distance irreducible elements: the moves that appear in all distance reducing Markov bases.
\end{abstract}

\maketitle

\tableofcontents

\section{Introduction}

In Algebraic Statistics, Markov bases play an important role in sampling contingency tables and approximating Fisher's exact test with a Markov chain Monte Carlo approach. For background on Algebraic Statistics we refer to \cite{sullivant2018algebraic, drton2009lectures, casanellas2020algebraic}, and for a thorough survey on Markov bases and contingency tables we refer to \cite{aoki2012markovbook}. When the sample size is large, it becomes infeasible to write down every possible contingency table. As an approximation, we may use a Markov chain Monte Carlo method outlined by Diaconis and Sturmfels \cite{diaconis1998algebraic}: start with any sample with the correct marginals; repeatedly apply random moves to obtain a random sample. The set of moves is a fixed Markov basis, which is a generating set of the underlying toric ideal. By definition, a Markov basis produces a connected Markov chain on the fiber of the contingency table. So, after a certain number of moves, the sample produced is close to the stationary distribution of the Markov chain. The number of moves required to achieve this is called the \textit{mixing time}.

It is highly desirable to have a Markov basis with a low mixing time. However, computing the mixing time is often very difficult. So, we would like to find Markov bases that tightly and predictably connect the fibers of the contingency table. This motivates the notion of a \textit{distance reducing} Markov basis. See Definition~\ref{def: distance reducing} and \cite{aoki2005distance}. For distance reducing Markov bases, the minimum number of moves required to connect a pair of points in a fiber is bounded above the distance between them.

There are many known examples of distance reducing Markov bases. For instance, the Graver basis is distance reducing with respect to the $1$-norm. For the homogeneous case, see \cite{aoki2005distance}, and for the general case we prove the following.

\begin{proposition*}[Proposition~\ref{prop: Graver is 1 norm reducing}]
    The Graver basis is strongly distance reducing, even if $A$ is inhomogeneous. 
\end{proposition*}

The family of \textit{strongly robust} toric ideals are those that are minimally generated by their Graver basis. While this may seem like a very strong condition, there are a number of families of examples that exhibit a rich combinatorial structure. See \cite{kosta2023strongly, kosta2024complex, petrovic2018bouquet, sullivant2019strongly, gross2013combinatorial}. For instance, the family of toric ideals of Lawrence type are strongly robust, see \cite[Chapter~7]{sturmfels1996grobner}. For each strongly robust toric ideal, any Markov basis is distance reducing. However, not all distance reducing Markov bases belong to strongly robust toric ideals.

\subsection{Our results}

Typically the Graver basis is very large and impractical to compute, even for some small examples. So it is useful to find methods to determine whether a Markov basis is distance reducing with very few computational steps. In this paper, we characterise the distance reduction property for complete intersection monomial curves in terms of the \textit{circuits}.

\begin{theorem*}[Corollary~\ref{cor: ci dist red iff dist red circuits}]
Let $A \in \ZZ^{1 \times n}$ and assume that $I_A$ is a complete intersection toric ideal. Let $M \subseteq \ker(A)$ be a minimal Markov basis. Then $M$ is distance reducing if and only if $M$ reduces the distance of the circuits of $A$.
\end{theorem*}

In the above, a circuit refers to an element $z \in \ker(A)$ of the form
\[
    z = (0, \dots, 0, z_i, 0, \dots, 0, -z_j, 0, \dots, 0),
\]
i.e., an element supported on exactly two coordinates. This classification shows that to check the distance reduction property for monomial curves, it suffices to check that a Markov basis reduces the distance of the circuits. The circuits form a very small subset of the Graver basis, so this result is a significant improvement over checking the entire Graver basis, see Theorem~\ref{thm: 1 norm reducing test via Graver basis} below.

For non complete intersection, we give a similar characterisation. For monomial curves in $\mathbb A^3$, we use Herzog's classification of Markov bases \cite{Herzog1970} to prove the following.

\begin{theorem*}[Theorems~\ref{thm: dim3 all}, \ref{thm: dim3 ci} and \ref{thm: dim3 nci}]
Let $A \in \ZZ^{1 \times 3}$ and $M$ a minimal Markov basis for $A$. Then $M$ is distance reducing if and only if $M$ reduces the distance of the circuits of $A$.
\end{theorem*}

Similarly for monomial curves in $\mathbb A^4$, we characterise the distance reduction property whenever the curve admits a \textit{gluing}.

\begin{theorem*}[Theorem~\ref{thm: dim4 nci glue}]
Let $A \in \ZZ^{1 \times 4}$ and $M \subseteq \ker(A)$ a minimal Markov basis for $A$. Assume that $A$ admits a gluing. Then $M$ is distance reducing if and only if $M$ reduces the distance of the circuits of $A$.
\end{theorem*}

A gluing is a operation that arises in the classification of complete intersection toric ideals \cite{rosales1997presentations, morales2005complete}. The proofs of the above results carefully use the properties of Markov bases that arise from glued monomial curves.

There are many important families of moves related to Markov bases. For instance: the \textit{indispensables} $S(A)$ are the moves that belong to every Markov basis; the \textit{Graver basis} $G(A)$ is a Markov basis given by the set of \textit{primitive} elements; and the \textit{Universal Markov basis} $M(A)$ is the union of all minimal Markov bases. These sets also admit algebraic descriptions in terms of certain decompositions \cite{charalambous2014markov}.

The distance-reducing analogue of the indispensables is the set of \textit{distance irreducible elements}, which are the elements that appear in every distance reducing Markov basis. In Section~\ref{sec: dist irred elements}, we define and investigate the relationships between the families of moves $S(A)$, $M(A)$, $G(A)$ above, and their distance-reducing analogues: weakly distance irreducibles $D^w(A)$, distance irreducibles $D(A)$, universal strongly distance reducing Markov basis $\MD^s(A)$, and universal distance-reducing Markov basis $\MD(A)$.

\begin{proposition*}[Proposition~\ref{prop: dist irred chain inclusions}]
The following chain of inclusions holds:
\[
    S(A) \subseteq D(A) \subseteq D^w(A) \subseteq G(A).
\]
\end{proposition*}

If $A$ has a unique minimal Markov basis then we may relate the universal Markov basis to the distance irreducible elements as follows.

\begin{proposition*}[Proposition~\ref{prop: unique min Markov dist irred sets}]
    If $A$ has a unique distance-reducing minimal Markov basis, then we have
    \[
    S(A) = D(A) = M(A) = \mathcal D(A).
    \]
\end{proposition*}

We also show that there are always finitely many minimal distance reducing Markov bases.

\begin{theorem*}[Corollary~\ref{cor: universal dist red bases are finite}]
    The sets $\MD(A)$ and $\MD^s(A)$ are finite.
\end{theorem*}

\subsection{Paper outline}
In Section~\ref{sec: prelim} we fix our notation for: Markov bases; distance reduction, see Definition~\ref{def: distance reducing}; and gluings in Section~\ref{sec: prelim ci and gluing}. 

Sections~\ref{sec: dim3 monomial curves}-\ref{sec: dim4 mon curves} are about the distance reduction property for monomial curves. In Section~\ref{sec: dim3 monomial curves}, we characterise the distance reduction property of monomial curves in $\mathbb A^3$. Our main results are Theorem~\ref{thm: dim3 ci} for complete intersections and Theorem~\ref{thm: dim3 nci} for non complete intersections. In Section~\ref{sec: monomial curves first kind} we characterise the distance reduction property for a special kind of symmetric complete intersection monomial curves given by Theorem~\ref{thm: characterisation dist red first kind}. In Section~\ref{sec: ci mon curves}, we prove Theorem~\ref{thm: dist red ci implies first kind}, which shows that distance reducing Markov bases of complete intersection monomial curves are exactly those in Section~\ref{sec: monomial curves first kind}. In Section~\ref{sec: dim4 mon curves}, we focus on monomial curves in $\mathbb A^4$. We make explicit the characterisation complete intersections in Corollary~\ref{cor: dim 4 markov basis type 211} and  characterise the distance reducing Markov bases of glued non complete intersections in Theorem~\ref{thm: dim4 nci glue}.

In Section~\ref{sec: generalisation of 1 norm results} we generalise results from \cite{aoki2005distance} to the inhomogeneous case and prove that a Markov basis is distance reducing if and only if it reduced the distance of the Graver basis. See Theorem~\ref{thm: 1 norm reducing test via Graver basis}. In Section~\ref{sec: dist irred elements}, we report on the \textit{distance irreducible elements}, which are the elements that appear in every distance reducing Markov basis. In Section~\ref{sec: discussion} we discuss possible generalisations of our results and further questions. In particular, in Section~\ref{sec: discussion dist red complex} we introduce the distance reducing complex for studying families of metrics.

\section{Preliminaries}\label{sec: prelim}

\noindent
\textbf{Notation.}
Throughout, we write $\NN := \{0, 1, 2, \dots \}$ for the set of non-negative integers and $[n] := \{1, 2, \dots, n \}$ for the integers from $1$ to $n$.
For any set $S$, such as $\NN$, $\ZZ$ or $\QQ$, we take elements of $S^n$ to be column vectors. 
Given a set of positive integers $X = \{x_1, x_2, \dots, x_n\}$, we write
$\NN X \subseteq \NN$ for the numerical semigroup generated by $X$ and $\ZZ X \subseteq \ZZ$ for the subgroup generated by $X$.

\smallskip
\noindent \textbf{Affine semigroups.}
Throughout, we consider matrices $A \in \ZZ^{d \times n}$, which, unless specified otherwise, we assume satisfy $\ker(A) \cap \NN^n = \{0\}$. For any vector $u \in \ZZ^n$ we define $u^+$ and $u^-$ to be the unique vectors in $\NN^n$ such that $u = u^+ - u^-$. We denote by $(u)_i$ is the $i$-th coordinate of $u$ and $u_i = |(u)_i|$ its absolute value. Let $R = K[x_1, \dots, x_n]$ be a polynomial ring and recall the multi-index notation for monomials $x^u := x_1^{u_1}x_2^{u_2} \cdots x_n^{u_n} \in R$ for any $u = (u_1, \dots, u_n)^T \in \NN^n$. The columns $a_1, \dots, a_n$ of $A$ generate the affine semigroup \[
\NN A := \left\{\sum_{i \in [n]} \lambda_i a_i \colon \lambda_i \in \NN \right\} \subseteq \NN^d.
\]
For each $t \in \NN A$, we write $\MF_t = \{ u \in \NN^n \colon Au = t \}$ for the \textit{fiber} of $A$. In addition, given any element $v \in \NN^n$, we write $\MF(v) := \MF_{Av}$ for the fiber containing $v$. We denote by $\MF(A) = \{\MF_t \colon t \in \NN A \}$ the collection of all fibers of $A$. 

\smallskip 
\noindent \textbf{Markov bases.}
The \textit{toric ideal} of $A$ is defined to be $I_A = \langle x^{u^+} - x^{u^-} \colon u \in \ker(A) \rangle \subseteq R$. A \textit{Markov basis} is a set $B \subseteq \ker(A)$ such that $I_A = \langle x^{u^+} - x^{u^-} \colon u \in B\rangle$. A Markov basis is \textit{minimal} if no proper subset is a Markov basis. The union of all minimal Markov bases is the \textit{Universal Markov basis}, denoted $M(A)$. The intersection of all Markov bases $S(A)$ is called the set of indispensable elements. The natural partial ordering on $\NN^n$ is given by $u \le v \iff u_i \le v_i$ for all $i \in [n]$. An element $z \in \ker(A)$ is called \textit{primitive} if $y \in \ker(A)$ with $y^+ \le z^+$ and $y^- \le z^-$ then $y = z$. The set of all primitive elements of $\ker(A)$ is the \textit{Graver basis} $G(A)$ of $A$.

\medskip
\noindent
\textbf{Decompositions.} Let $z \in \ker(A)$. A decomposition of $z$ is an expression for $z$ as a sum of elements of $\ker(A)$. We say that a decomposition is proper if no summands is zero. Below we describe the decompositions that appear throughout the paper.

\smallskip
\indent
\textit{Conformal decomposition:} $z = u+v$ such that $z^+ = u^+ + v^+$ and $z^- = u^- + v^-$.

\smallskip
\indent
\textit{Semi-conformal decomposition:} $z = u+v$ such that $u_i > 0  \implies v_i \ge 0$.

\smallskip
\noindent
The indispensable elements $S(A)$ and the elements of the Graver basis $G(A)$ are completely characterised by the above decompositions.

\begin{theorem}[{\cite{sturmfels1996grobner}, \cite[Proposition~1.1]{charalambous2014markov}}]
    We have the following:
    \begin{itemize}
        \item $z \in G(A)$ if and only if $z \in \ker(A)$ has no proper conformal decomposition,
        \item $z \in S(A)$ if and only if $z \in \ker(A)$ has no proper semi-conformal decomposition. 
    \end{itemize}
\end{theorem}

\medskip
\noindent 
\textbf{Moves inside fibers.}
Let $u \in \ker(A)$ be a nonzero element. It is convenient to think of $u$ as a \textit{move} in the following sense. Suppose that $x \in \NN^n$. We say that $u$ is \textit{applicable} to $x$ if either
$x^- \ge u^-$ or $x^+ \ge u^+$. If $x^- \ge u^-$ then we say $u$ sends $x$ to $x+u \in \NN^n$, otherwise if $x^+ \ge u^+$ then we say $u$ sends $x$ to $x-u \in \NN^n$. Given an element $z \in \ker(A)$, we say that a move \textit{$u$ is applicable to $z$} if $u$ is applicable to either $z^+$ or $z^-$.

Suppose that $x$ and $y$ belong to the same fiber of $A$. We say that a subset $B \subseteq \ker(A)$ \textit{connects} $x$ and $y$ if there is a sequence of moves $u = (u_1, u_2, \dots, u_k)$ such that for each $i \in [k]$ we have that with $u_i \in B$ is a move that sends 
\[
x + u_1 + \dots + u_{i-1} \quad \text{ to } \quad x + u_1 + \dots + u_i
\]
and $x+u_1+\dots +u_k = y$. The Fundamental Theorem of Toric Ideals states that a collection of moves $B \subseteq \ker(A)$ is a Markov basis if and only if $B$ connects any pair of points within each fiber of $A$.

\medskip
\noindent
\textbf{Metrics.} Let $X$ be a set. A \textit{metric} on $X$ is a function $d: X \times X \rightarrow \RR \cup \{\infty \}$ such that for all $x, y, z \in X$:
\[
    d(x,y) \ge 0, \quad
    d(x,y) = 0 \iff x = y, \quad
    d(x,y) = d(y,x), \quad
    d(x,y) + d(y,z) \ge d(x,z).
\]
Throughout this article, we often consider the metric on $\RR^n$ induced by the $1$-norm $|| \cdot ||$, which is given by $d(x,y) = ||x-y|| = \sum_i |x_i - y_i|$, where $|\cdot|$ is the absolute value on $\RR$. 

\smallskip

In their work, Aoki and Takemura identify a special class of Markov bases called \textit{distance reducing Markov bases}.

\begin{definition} \label{def: distance reducing}(Distance reducing)
    Fix a metric $d$ on $\NN^n$ and subsets $B, Z \subseteq \ker(A)$. We say that \textit{$B$ reduces the distance of $Z$} with respect to $d$ (or \textit{$B$ $d$-reduces $Z$}) if for each $z \in Z$, there exists $u \in B$ such at least one of the following holds:
    \begin{itemize}
        \item $u$ is applicable to $z^+$ and either $d(z^++u, z^-) < d(z^+,z^-)$ or $d(z^+-u, z^-) < d(z^+,z^-)$,
        \item $u$ is applicable to $z^-$ and either $d(z^+, z^-+u) < d(z^+,z^-)$ or $d(z^+, z^--u) < d(z^+,z^-)$.
    \end{itemize}
    We say that $B$ \textit{strongly reduces the distance of $Z$} with respect to $d$ if for each $z \in Z$ each of the above conditions is satisfied by some, possibly different, $u \in B$.
    We say that \textit{$B$ is [strongly] distance reducing} with respect to $d$ if $B$ [strongly] reduces the distance of $\ker(A)$ with respect to $d$.
    We say that \textit{$B$ [strongly] reduces the distance of $Z$} if $B$ [strongly] reduces the distance of $Z$ with respect to the metric induced by the $1$-norm.
    We say that \textit{$B$ is [strongly] distance reducing} if $B$ [strongly] reduces the distance of $\ker(A)$ with respect to the metric induced by the $1$-norm. 
\end{definition}

In other words, a set of moves $B$ is distance reducing if for any $x$, $y$ in the same fiber of $A$ there exists a move $u \in B$ that reduces the distance between them. For the $1$-norm, this means that $u$ is applicable to at least one of $x$ and $y$ either
\begin{equation}\label{eqn: 1-norm reduction}
    ||x-y-u|| < ||x-y|| 
    \quad \text{or} \quad
    ||x-y+u|| < ||x-y||.
\end{equation}
And $B$ is strongly distance reducing if there exist moves $u, u' \in B$ applicable to $x, y$ respectively that satisfy (\ref{eqn: 1-norm reduction}).

\begin{proposition}[Aoki-Takemura Proposition~1]\label{prop: B d-reducing => Markov basis}
    If $B$ is distance reducing with respect to some metric then $B$ is a Markov basis.
\end{proposition}

To illustrate the proof of this proposition we provide the following illuminating sketch. Take $x$ and $y$ in the same fiber of $A$ and consider the task of finding a sequence of moves in $B$ that connects $x$ and $y$. Since $B$ is distance reducing, we can find a move $u$, which is applicable to $x$ or $y$ and reduces the distance between them. We can continue in this way to find a move at each subsequent step that reduces the distance. Since there are only finitely elements in the fiber, we will eventually obtain a finite sequence of moves in $B$ that connects $x$ and $y$. Hence $B$ is a Markov basis.

\begin{remark}
    We note an important difference between our setup and that of Aoki-Takemura. We do not assume that $A$ is homogeneous, i.e., we do not assume that the all $1$'s row-vector is not a linear combination of the rows of $A$. In Section~\ref{sec: generalisation of 1 norm results}, we generalise some results about the homogeneous to the inhomogeneous case. 
\end{remark}

\medskip
\noindent \textbf{Minimal sets.}
A \textit{minimal Markov basis} is a Markov basis such that any proper subset is not a Markov basis. Similarly, a set is \textit{minimally distance reducing} if no proper subset is distance reducing. By Proposition~\ref{prop: B d-reducing => Markov basis}, any minimally distance reducing set is a Markov basis so we often call such a set a \textit{minimally distance reducing Markov basis}. If a minimal Markov basis is distance reducing then we call it a \textit{distance-reducing minimal Markov basis}. Note any that distance-reducing minimal Markov bases are minimally distance reducing Markov bases. However, a minimally distance reducing Markov basis need not be a minimal Markov basis.

\subsection{Complete intersections and gluings}\label{sec: prelim ci and gluing}

Let $A = \begin{pmatrix}
    a_1 & a_2 & \dots & a_n
\end{pmatrix} \in \ZZ^{d \times n}$ be a matrix of positive integers such that $\ker(A) \cap \NN^n = \{0\}$. We say that $A$ is a \textit{complete intersection} if the ideal $I_A$ is a complete intersection, i.e., $I_A$ is generated by $\nullity(A)$ elements.

The complete intersection matrices $A$  are characterised by \textit{gluings}. More precisely, the matrix $A$ is a complete intersection if and only if the columns of $A$ can be partitioned into two parts $\{a_1, a_2, \dots, a_n\} = \{b_1, \dots, b_p\} \sqcup \{c_1, \dots, c_q\}$ such that the matrices $B = \begin{pmatrix}
    b_1 & \dots & b_p
\end{pmatrix}$ and $C = \begin{pmatrix}
    c_1 & \dots & c_q
\end{pmatrix}$ are both complete intersections and satisfy the gluing property:
\[
\text{there exists }
x \in \NN B \cap \NN C
\text{ such that }
x\ZZ = \ZZ B \cap \ZZ C.
\]
Equivalently, there exists $z \in \ker([B \mid C])$ with $\supp(z^+) \subseteq \{1, \dots, p\}$ and $\supp(z^-) \subseteq \{p+1, \dots, p+q\}$ such that
    \[
    \ker([B \mid C]) = 
    \left(\ker(B) \times \{0\}\right) \oplus 
    \left(\{0\} \times \ker(C)\right) \oplus
    \langle z \rangle_\ZZ.
    \]
If $B \in \ZZ^{d \times p}$ and $C \in \ZZ^{d \times q}$ are matrices that satisfy the gluing property, then we write $B \circ C = [B \mid C]$ for the juxtaposition of $B$ and $C$.

The equivalence in the above definition is proved originally in \cite[Theorem~1.4]{rosales1997presentations} and, more generally in the setting of $p$-gluings, in \cite[Theorem~2.6]{morales2005complete}. The definitions are related as follows.
If $x \in \NN B \cap \NN C$, then $x = \sum_{i \in [p]} \lambda_i b_i = \sum_{j \in [q]} \mu_j c_j$ for some non-negative integers $\lambda_i$ and $\mu_j$, as in the first definition. Then the element $z = (\lambda_1, \dots, \lambda_p, -\mu_1, \dots, -\mu_q) \in \ker([B \mid C])$ satisfies the second definition. Similarly, if $z = (z_1, \dots, z_p, -z_{p+1}, \dots, -z_{p+q}) \in \ker([B \mid C])$ satisfied the second definition, then $x = b_1 z_1 + \dots + b_p z_p = c_1 z_{p+1} + \dots c_q z_{p+q} \in \NN A \cap \NN B$ satisfies the first definition.  

\medskip
\noindent \textbf{Gluing notation.} We introduce the following notation to keep track of gluings for monomial curves. Assume that $A = \begin{pmatrix}
    a_1 & a_2 & \dots & a_n
\end{pmatrix} \in \ZZ^{1 \times n}$ is a complete intersection. The \textit{gluing type} of $A$ is defined recursively. The base case is $n = 1$, which has gluing type $a_1$. If $A$ is the gluing of two complete intersections $B$ and $C$ with gluing types $T_B$ and $T_C$ respectively, then the type of $A$ is denoted $(T_B \circ T_C)$. For extra detail, we may also record the value $x \in \NN B \cap \NN C$ for the gluing and write the gluing type of $A$ as $(T_B \circ_x T_C)$.  So, for example, all complete intersection monomial curves in $\mathbb A^3$ have gluing type $((a_1 \circ a_2) \circ a_3)$ for an appropriate ordering of the coordinates. By a slight abuse of notation we identify $A$ with its gluing type.

\begin{example}
    Some matrices admit different gluing types. Let $A = \begin{pmatrix}
        3 & 5 & 9
    \end{pmatrix}$ then we have
    \[
    A = ((3 \circ_{15} 5) \circ_{9} 9)
    \quad \text{and} \quad
    A = ((3 \circ_9 9) \circ_{15} 5).
    \]
    There are two distinct minimal Markov bases for $A$ given by
    \[
        M_1 = \begin{bmatrix}
            5 & -3 & 0 \\
            3 & 0  & -1
        \end{bmatrix}
        \quad \text{and} \quad
        M_2 = \begin{bmatrix}
            2 & -3 & 1 \\
            3 & 0  & -1
        \end{bmatrix}.
    \]
    We observe that $M_1$ is not distance reducing as it fails to reduce the distance of the circuit $(0, 9, -5)$. On the other hand, it turns out that $M_2$ is distance reducing, which follows from Theorem~\ref{thm: dim3 ci}.
\end{example}

\section{Monomial curves in \texorpdfstring{$\mathbb A^3$}{A3}}\label{sec: dim3 monomial curves}

In this section we characterise the distance-reducing minimal Markov bases of monomial curves in $\mathbb A^3$. Throughout this section, we consider a matrix $A = \begin{pmatrix}a_1 & a_2 & a_3 \end{pmatrix}$ with distinct entries. 

\begin{theorem}\label{thm: dim3 all}
     Let $M$ be a minimal Markov basis for $A$. Then $M$ is distance reducing if and only if $M$ reduces the distance of the circuits of $A$.
\end{theorem}

\begin{proof}
    If $A$ is a complete intersection then the result follows from Theorem~\ref{thm: dim3 ci}. Otherwise, $A$ is not a complete intersection and the result follows from Theorem~\ref{thm: dim3 nci}.
\end{proof}

To prove Theorems~\ref{thm: dim3 ci} and \ref{thm: dim3 nci}, we use the explicit characterisation of the minimal Markov bases for monomial curves in $\mathbb A^3$. Let us begin by recalling from \cite{Herzog1970} the following terminology for elements of $\ker(A)$.

\medskip
\noindent \textbf{Notation.} Let $z \in \ker(A)$. If $z$ is nonzero, then there is a coordinate of $z$ of a different sign to the other coordinates, i.e., for some $i \in [3]$ we have that either
\begin{itemize}
    \item[1.] $z_i > 0$ and $z_j \le 0$ for each $j \neq i$, 
    \item[2.] $z_i < 0$ and $z_j \ge 0$ for each $j \neq i$.
\end{itemize}
In this case, we say that \textit{$z$ has type $i$}. Note, if $z$ is a circuit, then $z$ has two types. Suppose that $z$ has type $i$ and for any other element $z' \in \ker(A)$ of type $i$ we have $|z_i| \le |z'_i|$, then we say that $z$ is \textit{minimal type $i$}, or \textit{minimal} when the context is clear.

\smallskip

In \cite{Herzog1970}, Herzog shows that the ideal $I_A$ is minimally generated by the binomials corresponding to minimal elements of $\ker(A)$. Moreover, these generating sets directly determine whether $A$ is a complete intersection. The following theorem summarises the important results that we use throughout this section.

\begin{theorem}[{\cite[Propositions~3.3 and 3.5 and Theorem~3.8]{Herzog1970}}]\label{thm: Herzog dim 3}
    The minimal Markov bases of $A$ fall into one of two cases.
    \begin{enumerate}[parsep=5pt]
        \item \label{case: non ci thm Herzog dim 3}
        If no minimal element of $\ker(A)$ is a circuit, then, up to sign, there are unique minimal elements of $\ker(A)$:
        \[
        g_1 = (-c_1,   v_{12}, v_{13}), \ 
        g_2 = (v_{21}, -c_2,   v_{23}), \ 
        g_3 = (v_{31}, v_{32}, -c_3)
        \]
        of types $1$, $2$, $3$, respectively, with $c_i$ and $v_{i,j}$ positive for all $i, j$.
        Up to sign, $A$ has a unique minimal Markov basis that consists of the three minimal elements. In this case, $A$ is not a complete intersection.
        
        \item \label{case: ci thm Herzog dim 3} If a circuit is a minimal element of $\ker(A)$, then, up to sign and permutation of the coordinates, two  minimal elements of $\ker(A)$ are
        \[
        b = (b_1, -b_2, 0) \text{ and } c = (c_1, c_2, -c_3)
        \]
        where $b$ is the unique minimal type $1$ and $2$ element, and $c$ is a minimal type $3$ element, with $b_1, b_2, c_3 > 0$ and $c_1, c_2 \ge 0$. The other minimal elements of $\ker(A)$ are $c + \lambda b$ for any $\lambda \in \ZZ$ such that $c + \lambda b$ has type $3$. The minimal Markov bases are given by $\{b, c + \lambda b\}$ for each $\lambda$ as above. In this case, $A$ is a complete intersection.
    \end{enumerate}
\end{theorem}

If $A$ is not a complete intersection, then the unique minimal Markov basis described above satisfies the following relation.

\begin{proposition}[{\cite[Proposition~3.2]{Herzog1970}}]\label{prop: 3 dim nci gi sum to zero}
    If $A$ is not a complete intersection, then its unique minimal Markov basis $\{g_1, g_2, g_3 \}$ satisfies $g_1 + g_2 + g_3 = 0$.
\end{proposition}

By Theorem~\ref{thm: dim3 all}, the distance reduction property is characterised by circuits. We use this to observe that a unique minimal Markov basis does not imply that it is distance reducing.

\begin{example} \label{example: 3-dim CI: unique MB does not imply reducing}
    Let $A = \begin{pmatrix}
        3 & 5 & 11
    \end{pmatrix}$. Then $A$ has a unique minimal Markov basis 
    \[
    M = \{b := (5, -3, 0), \,  c:= (2, 1, -1)\},
    \]
    which is evident from the observation that $b$ is not applicable to $c$. Hence, there is not nonzero $\lambda \in \ZZ$ such that $c + \lambda b$ has type $3$. Consider the circuit $z = (0, 11, -5) \in \ker(A)$. The elements $b$ and $c$ are applicable to $z$. However, neither reduces the distance of $z$ as
    \[
        ||z + b|| = ||(5, 8, -5)|| = 18 \not < ||z||
        \text{ and }
        ||z - c|| = ||(-2, 10, -4)|| = 16 \not < ||z||.
    \]

\end{example}

\subsection{Complete intersections}

By Theorem~\ref{thm: Herzog dim 3}-(\ref{case: ci thm Herzog dim 3}), the complete intersection monomial curves in $A = \begin{pmatrix}
    a_1 & a_2 & a_3
\end{pmatrix}$ have a minimal Markov basis consisting of a circuit $b = (b_1, -b_2, 0)$ and another element $c = (c_1, c_2, -c_3)$. By swapping the first and second column of $A$, we may assume that $b_1 > b_2$ or, equivalently, that $a_1 < a_2$. 

\begin{theorem}\label{thm: dim3 ci}
    Suppose $A$ is a complete intersection and let 
    \[
    M = \{b := (b_1, -b_2, 0),\ c := (c_1, c_2, -c_3)\}
    \]
    be a minimal Markov basis for $A$, as in Theorem~\ref{thm: Herzog dim 3}, with $b_1 > b_2$. Then the following are equivalent:
    \begin{enumerate}
        \item $M$ is distance reducing,
        \item $M$ reduces the distance of the circuit $z = \gamma(0, a_3, -a_2)$ where $\gamma = 1 / \gcd(a_2, a_3)$,
        \item $c_1 < c_2 + c_3$.
    \end{enumerate}
\end{theorem}

\begin{proof}
$(1) \implies (2)$ Immediate from the definition.

$(2) \implies (3)$ Assume that $M$ reduces the distance of the circuit $z = (0, z_2, -z_3)$. If $b$ is applicable to $z$ then it does not reduce the distance of $z$ because $b_1 > b_2$. So $c$ reduces the distance of $z$. Since $z \in \ker(A)$ and $M$ is a Markov basis, we have $z = -\alpha b + \beta c$ for some integers $\alpha, \beta$. So
\[
z = (0, z_2, -z_3) = (-\alpha b_1 + \beta c_1,\ 
\alpha b_2 + \beta c_2,\ 
-\beta c_3).
\]
Since the third coordinate of $b$ is zero, it follows that $\beta > 0$. By the first coordinate, we have $-\alpha b_1 + \beta c_1 = 0$, so it follows that $\alpha \ge 0$. Hence $z_2 = \alpha b_2 + \beta c_2 \ge c_2$. Also note that $z$ is type $3$ and $c$ is minimal type $3$, so we have $c_3 \le z_3$. Since $c$ reduces the distance of $z$, we have $||z|| > ||z - c||$, hence
\[
||z|| > ||z - c|| = |-c_1| + |z_2 - c_2| + |-z_3 + c_3| = ||z|| + c_1 - c_2 - c_3. 
\]
So we have $c_1 < c_2 + c_3$.

\bigskip

\noindent
$(3) \implies (1)$ Assume that $c_1 < c_2 + c_3$. Let $z \in \ker(A)$ by any element. We show that $M$ reduces the distance of $z$. We write $z = \alpha b + \beta c = (\alpha b_1 + \beta c_1, -\alpha b_2 + \beta b_2, -\beta c_3)$ for some integers $\alpha$, $\beta$. If either $\alpha = 0$ or $\beta = 0$, then $z$ is a multiple of $c$ or $b$, respectively, so $c$ or $b$ is applicable to $z$ and reduces the distance. So, from now on, we assume $\alpha \neq 0$ and $\beta \neq 0$. If $\beta < 0$ then we replace $z$ with $-z$. So without loss of generality, we may assume that $\beta > 0$.

If $\alpha > 0$ then $z_1 = \alpha b_1 + \beta b_2 \ge b_1$, hence $b$ is applicable to $z$. So we have
\[
||z - b|| = ||(z_1 - b_1, \pm z_2 + b_2, -z_3)|| \le z_1 - b_1 + z_2 + b_2 + z_3 = ||z|| - b_1 + b_2 < ||z||.
\]
Hence $b$ reduces the distance of $z$ and we are done.

If $\alpha < 0$, then we have $z_2 = (-\alpha)b_2 + \beta b_2 \ge c_2$. Since $z_3 = \beta c_3 \ge c_3$, so $c$ is applicable to $z$. Hence
\[
||z-c|| = ||(\pm z_1 - c_1, z_2 - c_2, -z_3+c_3)||
= |z_1 \pm c_1| + z_2 - c_2 + z_3 - c_3
\le ||z|| + c_1 - c_2 - c_3
< ||z||,
\]
where the final inequality follows from the assumption that $c_1 < c_2 + c_3$.
\end{proof}

\subsection{Non complete intersections} The minimal Markov basis for a non complete intersection monomial curves in $\mathbb A^3$ is explicitly described in Theorem~\ref{thm: Herzog dim 3}-(\ref{case: non ci thm Herzog dim 3}). With this description, we classify when the Markov basis is distance reducing as follows.

\begin{theorem}\label{thm: dim3 nci}
    Let $A = \begin{pmatrix}
        a_1 & a_2 & a_3
    \end{pmatrix}$ be a matrix with $a_1 < a_2 < a_3$ that is not a complete intersection. Let 
    \[
    M = \{g_1 := (-c_1, v_{12}, v_{13}),\ 
    g_2 := (v_{21}, -c_2, v_{23}),\ 
    g_3 := (v_{31}, v_{32}, -c_3)\}
    \]
    be a minimal Markov basis for $A$, as in Theorem~\ref{thm: Herzog dim 3}. Then the following are equivalent:
    \begin{enumerate}
        \item $M$ is distance reducing,
        \item $M$ reduces the distance of the circuit $z = \gamma(0, a_3, -a_2)$ where $\gamma = 1/ \gcd(a_2, a_3)$,
        \item at least one of $v_{21} < c_2 + v_{23}$ or $v_{31} < v_{32} + c_3$ holds.
    \end{enumerate}
\end{theorem}

\begin{proof}
$(1) \implies (2)$. Follows immediately from the definition.

$(2) \implies (3)$. Assume that $M$ reduces the distance of the circuit $z = (0, z_2, -z_3)$.
We have that $z = \alpha g_1 + \beta g_2 + \gamma g_3$ for some $\alpha, \beta, \gamma \in \ZZ$. By Proposition~\ref{prop: 3 dim nci gi sum to zero}, we have that $g_1 + g_2 + g_3 = 0$. Therefore $z = (\beta - \alpha)g_2 + (\gamma - \alpha)g_3$. So, without loss of generality, we may assume that $\alpha = 0$. We have
\[
z = \beta g_2 + \gamma g_3 = 
(\beta v_{21} + \gamma v_{31},\ 
-\beta c_2 + \gamma v_{32},\ 
\beta v_{23} - \gamma c_3).
\]
By comparing the sign patterns, we deduce that $\beta < 0$ and $\gamma > 0$. By assumption, $M$ reduces the distance of $z$. However, the move $g_1$ is not applicable to $z$. So we proceed by taking cases on whether $g_2$ or $g_3$ reduces the distance of $z$.

\medskip

\noindent\textbf{Case 1.} Assume that $g_2$ reduces the distance of $z$. Then we have that $g_2$ is applicable to $z$ and so $c_2 \le z_2$. Observe that $z_3 = (-\beta) v_{23} + \gamma c_3 \ge v_{23}$. Since $g_2$ reduces the distance of $z$, we have $||z|| > ||z + g_2||$ and so
\[
||z|| > ||z+g_2|| = ||(v_{21}, z_2 - c_2, -z_3 + v_{23})||
= v_{21} + z_2 - c_2 + z_3 - v_{23}
= ||z|| + v_{21} - c_2 - v_{23}.
\]
So we have that $v_{21} < c_2 + v_{23}$ and we are done.

\medskip

\noindent\textbf{Case 2.} Assume that $g_3$ reduces the distance of $z$. Then we have that $g_3$ is applicable to $z$ and so $c_3 \le z_3$. Observe that $z_2 = (-\beta)c_2 + \gamma v_{32} \ge v_{32}$. Since $g_3$ reduces the distance of $z$, we have $||z|| > ||z - g_3||$ and so
\[
||z|| > ||z-g_3|| = ||(-v_{31}, z_2 - v_{32}, -z_3 + c_3)||
= v_{31} + z_2 - v_{32} + z_3 - c_3
= ||z|| + v_{31} - v_{32} - c_3.
\]
So we have that $v_{31} < v_{32} + c_3$ and this concludes the proof of the forward direction.

\medskip
$(3) \implies (1)$. Assume that either $v_{21} < c_2 + v_{23}$ or $v_{31} < v_{32} + c_3$ holds. Let $z \in \ker(A)$ be any element. We show that $M$ reduces the distance of $z$. Since $z \in \ker(A)$, we have that $z = \alpha g_1 + \beta g_2 + \gamma g_3$ for some integers $\alpha, \beta, \gamma$. By Proposition~\ref{prop: 3 dim nci gi sum to zero}, we have $g_1 + g_2 + g_3 = 0$ and so we may assume that $\alpha = 0$. So we have
\[
z = \beta g_2 + \gamma g_3 = 
(\beta v_{21} + \gamma v_{31},\ 
-\beta c_2 + \gamma v_{32},\ 
\beta v_{23} - \gamma c_3).
\]
Note that if either $\beta = 0$ or $\gamma = 0$, then we have that $z$ is multiple of $g_3$ or $g_2$, respectively, so that element reduces the distance of $z$. So, from now on, we assume that $\beta \neq 0$ and $\gamma \neq 0$. Without loss of generality, we may assume that $\beta > 0$.

If $\gamma > 0$ then we have $z_1 = \beta v_{21} + \gamma v_{31} \ge v_{21} + v_{31} = c_1$, where the final equality follows from Proposition~\ref{prop: 3 dim nci gi sum to zero}. So $g_1$ is applicable to $z$. Since $a_1 < a_2 < a_3$, we have that
\[
c_1 = \frac{a_2}{a_1} v_{12} + \frac{a_3}{a_1} v_{13} > v_{12} + v_{13}.
\]
So we have
\[
||z+g_1|| = ||(z_1 - c_1, \pm z_2 + v_{12}, \pm z_3 + v_{13})|| \le ||z|| - c_1 + v_{12} + v_{13} < ||z||.
\]
Hence $g_1$ reduces the distance of $z$ and we are done.

It remains to consider the case when $\gamma < 0$. We proceed by taking cases based on which assumption, either $v_{21} < c_2 + v_{23}$ or $v_{31} < v_{32} + c_3$, holds.

\medskip

\noindent\textbf{Case 1.} Assume $v_{21} < c_2 + v_{23}$. Then we have $z_2 = \beta c_2 + (-\gamma)v_{32} \ge c_2$, hence $g_2$ is applicable to $z$. Note that $z_3 = \beta v_{23} + (-\gamma)c_3 \ge v_{23}$. So we have
\[
||z-g_2|| = ||(\pm z_1 - v_{21}, -z_2 + c_2, z_3 - v_{23})|| \le ||z|| + v_{21} - c_2 - v_{23} 
< ||z||.
\]
Hence $g_2$ reduces the distance of $z$ and we are done.

\medskip

\noindent\textbf{Case 2.} Assume $v_{31} < v_{32} + c_3$. Then we have $z_3 = \beta v_{23} + (-\gamma)c_3 \ge c_3$, hence $g_3$ is applicable to $z$. Note that $z_2 = \beta c_2 + (-\gamma)v_{32} \ge v_{32}$. So we have
\[
||z+g_3|| = ||(\pm z_1 + v_{31}, -z_2 + v_{32}, z_3 - c_3)|| \le ||z|| + v_{31} - v_{32} - c_3 
< ||z||.
\]
Hence $g_3$ reduces the distance of $z$, which concludes the proof.
\end{proof}

\section{Monomial curves of the first kind}\label{sec: monomial curves first kind}

Throughout this section, we consider a complete intersection monomial curve in $\mathbb A^n$ given by the matrix 
$A = \begin{pmatrix}
    a_1 & a_2 & \dots & a_n
\end{pmatrix}
\in \ZZ^{1 \times n}
$
such that $\ker(A) \cap \NN^n = \{0\}$. We assume that $A$ admits a \textit{gluing of the first kind}:
\[
A = (( \dots ((a_1 \circ a_2) \circ a_3) \dots ) \circ a_n).
\]
We call these \textit{monomial curves of the first kind}.
For each $k \in [n]$ we define the submatrix $A_k = \begin{pmatrix}
    a_1 & \dots & a_k
\end{pmatrix}$. Note that $A_k$ admits a gluing of the first kind
\[
A_k = (( \dots ((a_1 \circ a_2) \circ a_3) \dots ) \circ a_k).
\]
If $n > 1$, then $A_n = (A_{n-1} \circ a_n)$, hence any minimal Markov basis for $A_n$ is built from a Minimal Markov basis of $A_{n-1}$ as follows. Let $M_n$ be a minimal Markov basis for $A_n$ then there exists a minimal Markov basis $M_{n-1}$ for $A_{n-1}$ such that
\[
M_n = (M_{n-1}
 \times \{0\}) \cup 
\{u_n := 
\begin{pmatrix}
    u_{n,1} & u_{n,2} & \dots & u_{n, n-1} & -u_{n, n}
\end{pmatrix}
\}
\]
for some non-negative integers $u_1, \dots, u_n$.
So, inductively, we may identify a minimal Markov basis $M_n$ for $A_n$ with the rows of a matrix
\[
M_n = \begin{bmatrix}
    u_{2,1} & -u_{2,2} & 0 & 0 & \dots & 0 & 0 \\
    u_{3,1} & u_{3,2} & -u_{3,3} & 0 & \dots & 0 & 0 \\
    u_{4,1} & u_{4,2} & u_{4,3} & -u_{4,4} & \dots & 0 & 0 \\
    \vdots & \vdots & \vdots & \vdots & \ddots & \vdots & \vdots \\
    u_{n-1, 1} & u_{n-1, 2} & u_{n-1, 3} & u_{n-1, 4} & \dots & -u_{n-1, n-1} & 0 \\
    u_{n, 1} & u_{n, 2} & u_{n, 3} & u_{n, 4} & \dots & u_{n, n-1} & -u_{n,n}
\end{bmatrix}.
\]
Without loss of generality, we assume that $a_1 < a_2$.

\begin{remark}
    For any matrix $A \in \ZZ^{1 \times n}$, we say that the numerical semigroup $\NN A$ is \textit{symmetric} if there exists $m \in \NN$ such that $x \in \NN A$ if and only if $m - x \notin S$ for all $x \in \ZZ$. Note that this definition does not depend on a choice of generators of $\NN A$ or an ordering on them. We say that $A$ is \textit{specially symmetric} if for each $k \in [n-1]$ we have 
    $
    \lcm(\gcd(a_1, a_2, \dots, a_k), a_{k+1}) \in \NN A_k
    $.   
    By \cite[Proposition~2.1]{Herzog1970}, it follows that if $A$ is specially symmetric then $\NN A$ is symmetric.

    Observe that $\lcm(\gcd(a_1, \dots, a_k), a_{k+1})$ is the generator of the group $\ZZ A_k \cap \ZZ \{a_{k+1}\}$. By the definition of gluing, we have that $\lcm(\gcd(a_1, \dots, a_k), a_{k+1}) \in \NN A_k$ is satisfied if and only if there is a gluing $A_{k+1} = (A_k \circ a_k)$. So monomial curves of the first kind correspond to a special kind of symmetric numerical semigroup. 
\end{remark}

We now define a set of inequalities, which we use to completely characterise whether $M$ is distance reducing. In the following definition, we define conditions on the entries $u_{i,j}$ of $M$.

\begin{definition}\label{def: first kind Rij conditions}
Fix $2 \le i < j \le n$. We say $M$ satisfies condition $R_{i,j}$ if at least one of the following conditions holds:
\begin{itemize}
    \item[(i)] $u_{i,1} + u_{i,2} + \dots + u_{i,i-1} < u_{i,i}$,
    \item[(ii)] There exists $\ell \in \{i+1, i+2, \dots, j-1\}$ such that $u_{\ell,m} = 0$ for all $m \in [\ell - 1] \setminus \{i\}$, i.e. $u_\ell$ is a circuit supported on coordinates $i$ and $\ell$, and $u_{\ell,i} > u_{\ell,\ell}$,
    \item[(iii)] $\sum_{k \in [j]} u_{j, k} < 2(u_{j,i} + u_{j,j})$,
\end{itemize}
If $M$ satisfies condition ($s$) for some $s \in \{{\rm i},\, {\rm ii},\, {\rm iii}
\}$ but $i, j$ are not clear from context then we write $M$ satisfies condition $(i,j)$-($s$).
\end{definition}

We now state the main result of this section, which characterises the distance-reducing minimal Markov bases for monomial curves of the first kind. 

\begin{theorem}\label{thm: characterisation dist red first kind}
    Let $M$ be a minimal Markov basis for $A$. Then the following are equivalent:
    \begin{itemize}
        \item[(1)] $M$ is distance reducing,
        \item[(2)] $M$ is distance reducing for the circuits of $A$,
        \item[(3)] $M$ satisfies $R_{i,j}$ for every $2 \le i < j \le n$.
    \end{itemize}
\end{theorem}

\begin{proof}
    By definition it follows that $(1) \implies (2)$. By Theorem~\ref{thm: M dist red implies Rij} we have that $(2) \implies (3)$. By Theorem~\ref{thm: Rij implies M dist red} we have $(3) \implies (1)$, which concludes the proof.
\end{proof}

\begin{theorem}\label{thm: M dist red implies Rij}
    If $M$ reduces the distance of the circuits of $A$ then $M$ satisfies $R_{i,j}$ for every $2 \le i < j \le n$.
\end{theorem}

\begin{proof}
    Throughout the proof, we write
    \[
        z_{a, b} = (0, \dots, 0, z_a, 0, \dots, 0, -z_b, 0, \dots, 0)
    \]
    for the circuit of $A$ supported on $a, b \in [n]$, with $a < b$. Fix $2 \le i < j \le n$. We show that $M$ satisfies $R_{i,j}$. In addition, we prove that the circuit
    \[
        z := z_{i,j} = (0, \dots, 0, z_i, 0, \dots, -z_j, 0, \dots, 0)
    \]
    is distance reduced by some element $u_r$ with $r \le j$. We do this by reverse induction on $j$. The base case with $j = n$ is trivial.

    For the induction step, for each $m > j$ and for all $i' < m$, we assume that there exists some $i' \le r \le m$ such that $u_r$ reduces the distance of $z_{i',m}$. We call this the first induction hypothesis.
    By assumption, $M$ reduces the distance of $z$, so there exists $u_\ell$ that reduces the distance of $z$. Note that if $\ell < i$, then $u_\ell$ is not applicable to $z$, so we must have $\ell \ge i$. If $\ell \le j$ then we are done. So let us assume that $\ell > j$.

    Since $u_\ell$ is applicable to $z$, it follows that $\supp(u_\ell^+) = \{i\}$ or $\supp(u_\ell^+) = \{j\}$. Hence $u_\ell$ is a circuit 
    \[
        u_{\ell} = (0, \dots, 0, u_{\ell,k}, 0, \dots, 0, -u_{\ell, \ell}, 0, \dots, 0)
    \]
    for some $k \in \{i, j\}$. Since $u_\ell$ reduces the distance of $z$, it follows that $u_{\ell, k} > u_{\ell, \ell}$. 

    \medskip \noindent \textbf{Claim 1.} For each $k < m < \ell$, we have $(m, \ell)$-(i) holds.
    
    \begin{proof}
    We now prove the claim by reverse induction on $m$. For the base case, let $m = \ell-1$ and consider the circuit
    \[
        z_{\ell-1, \ell} = (0, \dots, 0, z_{\ell-1}, -z_{\ell}, 0, \dots, 0).
    \]
    By the first induction hypothesis, we have that either $u_\ell$ or $u_{\ell-1}$ reduces the distance of $z_{\ell-1, \ell}$. Since $u_\ell$ is a circuit with $u_{\ell,k} > u_{\ell,\ell}$, it follows that $u_{\ell}$ does not reduce the distance of $z_{\ell-1, \ell}$. So $u_{\ell-1}$ reduces the distance of $z_{\ell-1, \ell}$. It follows that
    \[
    u_{\ell-1, 1} + u_{\ell-1, 2} + \dots u_{\ell-1, \ell-2} < u_{\ell-1, \ell-1}
    \]
    hence $(\ell-1, \ell)$-(i) holds.

    For the induction step, fix $k < m < \ell-1$ and assume that $(p, \ell)$-(i) holds for every $p \in \{m+1, \dots, \ell-1 \}$. We call this the second induction hypothesis. By assumption $M$ reduces the distance of the circuit
    \[
        z_{m, \ell} = (0, \dots, 0, z_m, 0, \dots, 0, -z_\ell, 0, \dots, 0).
    \]
    By the first induction hypothesis, there exists there exists $r \in \{m, m+1, \dots, \ell\}$ such that $u_r$ reduces the distance of $z_{m, \ell}$. Since $u_\ell$ is a circuit with $u_{\ell,k} > u_{\ell,\ell}$, it follows that $r \neq \ell$. Suppose by contradiction that $r > m$. Since $u_r$ is applicable to $z_{m, \ell}$, it follows that $u_r$ is a circuit 
    \[
        u_r = (0, \dots, 0, u_{r,m}, 0, \dots, 0, -u_{r,r}, 0, \dots, 0)
    \]
    with $u_{r,m} > u_{r,r}$. However, by the second induction hypothesis, we have 
    $u_{r,m} < u_{r,r}$, which is a contradiction. So it follows that $r = m$. Since $u_m$ reduces the distance of $z_{m, \ell}$, it follows that
    \[
        u_{m,1} + u_{m,2} + \dots + u_{m,m-1} < u_{m,m}.
    \]
    Hence $(m,\ell)$-(i) holds. This concludes the proof of the claim.
    \end{proof}
    
    We proceed by taking cases on $k \in \{i, j\}$.

    \medskip \noindent \textbf{Case 1.} Assume $k = i$. Since $k+1 \le j$, by Claim~1, we have that $(j,\ell)$-(i) holds. So
    $u_{j,1} + u_{j,2} + \dots + u_{j,j-1} < u_{j,j}$.
    We now show that $(i,j)$-(iii) holds by contradiction. Assume that
    \[
        u_{j,1} + \dots + u_{j,i-1} + u_{j,i+1} + \dots + u_{j,j-1} \ge u_{j,i} + u_{j,j}. 
    \]
    So we have 
    \[
        2u_{j,i} + u_{j,j} \le u_{j,1} + \dots + u_{j,j-1} < u_{j,j}.
    \]
    Hence $2u_{j,i} < 0$, which is a contradiction. So we have shown that $(i,j)$-(iii) holds. Hence $M$ satisfies $R_{i,j}$. It follows easily that $u_j$ reduces the distance of $z$, which proves the induction step.

    \medskip \noindent \textbf{Case 2.} Assume $k = j$. We show the following claim.
    
    \medskip \noindent \textbf{Claim 2.} For each $m \in \{i, i+1, \dots, j-1\}$ either: there exists a circuit $u_s$ for some $m < s \le j$ with
    \[
        u_s = (0, \dots, 0, u_{s,m}, 0, \dots, 0, -u_{s,s}, 0, \dots, 0)
    \]
    such that if $(m, s) \neq (i,j)$ then $u_{s,m} > u_{s,s}$; or 
    $(m, j)$-(i) holds.
    
    \begin{proof}
    We proceed by reverse induction on $m$. For the base case let $m = j-1$ and consider the circuit
    \[
        z_{j-1, \ell} = (0, \dots, 0, z_{j-1}, 0, \dots, 0, -z_{\ell},0, \dots, 0).
    \]
    By the first inductive hypothesis, we have that $z_{j-1, \ell}$ is distance reduced by some element $u_r$ with $j-1 \le r \le \ell$. Since $u_\ell$ is a circuit with $u_{\ell, j} > u_{\ell, \ell}$, it follows that $u_\ell$ does not reduce the distance of $z_{j-1, \ell}$. Assume by contradiction that $j+1 \le r < \ell$, then we have that $u_r$ is a circuit
    \[
        u_r = (0, \dots, 0, u_{r,j-1}, 0, \dots, 0, -u_{r,r}, 0, \dots, 0)
    \]
    with $u_{r,j-1} > u_{r,r}$. However, by Claim~1, we have $(r,\ell)$-(i) holds, which gives us $u_{r,j-1} < u_{r,r}$, which is a contradiction. So we have that $r \in \{j-1, j\}$. If $r = j$, then $u_j$ is a circuit
    \[
        u_j = (0, \dots, 0, u_{j,j-1}, -u_{j,j}, 0, \dots, 0).
    \]
    Since $u_r$ reduces the distance of $z_{j-1, \ell}$, if $i < j-1$, then it follows that $u_{j,j-1} > u_{j,j}$, as desired. Otherwise, if $r = j-1$, then it immediately follows that $(j-1,j)$-(i) holds. This concludes the proof of the base case.

    For the induction step, fix $i \le m < j-1$ and assume that the claim holds for all $p \in \{m+1, \dots, j-1 \}$.   
    Consider the circuit
    \[
        z_{m, \ell} = (0, \dots, 0, z_m, 0, \dots, 0, -z_{\ell}, 0, \dots, 0).
    \]
    By assumption, $M$ reduces the distance of the circuit $z_{m, \ell}$. By the first inductive hypothesis $z_{m, \ell}$ is distance reduced by $u_r$ for some $m \le r \le \ell$. Recall that $u_\ell$ is a circuit with $u_{\ell,j} > u_{\ell,\ell}$, so $u_\ell$ does not reduce the distance of $z_{m, \ell}$, hence $r \neq \ell$.
    
    Let $P = \{p \in \{m+1, \dots, j-1\} : (p, j)$-(i) holds$ \}$ and $P^c = \{m+1, \dots, j-1\} \setminus P$ be its complement. By Claim~1, we have that $(p, \ell)$-(i) holds for all $p \in \{j+1, \dots, \ell-1\}$. Suppose that for some $p \in P \cup \{j+1, \dots, \ell-1 \}$, the move $u_p$ is applicable to $z_{m, \ell}$. Then it follows that $u_p$ is a circuit
    \[
        u_p = (0, \dots, 0, u_{p, m}, 0, \dots, 0, -u_{p,p}, 0, \dots, 0).
    \]
    Since one of $(p, j)$-(i) or $(p,\ell)$-(i) holds, it follows that $u_{p,m} < u_{p,p}$. So $u_p$ does not reduce the distance of $z_{m, \ell}$. Hence $r \neq p$. So we have that $r \in P^c \cup \{m, j\}$. 
    
    If $r = j$, then it follows that $u_j$ is a circuit
    \[
        u_j = (0, \dots, 0, u_{j,m}, 0, \dots, 0, -u_{j,j}, 0, \dots, 0).
    \]
    Since $u_j$ reduces the distance of $z_{m, \ell}$, if $m > i$ then we have $u_{j,m} > u_{j,j}$, so the claim holds. 
    
    If $r = m$ then it follows that $(m,j)$-(i) holds.

    Suppose that $r \in P^c$. Then it follows that $u_r$ is a circuit
    \[
        u_r = (0, \dots, 0, u_{r,m}, 0, \dots, 0, -u_{r,r}, 0, \dots, 0).
    \]
    Since $u_r$ reduces the distance, we must have that $u_{r,m} > u_{r,r}$, which concludes the proof of the claim.
    \end{proof}

    So by Claim~2, with $m = i$, either: there exists a circuit $u_s$ for some $i < s \le j$ with
    \[
        u_s = (0, \dots, 0, u_{s,i}, 0, \dots, 0, -u_{s,s}, 0, \dots, 0),
    \]
    such that if $s \neq j$ then $u_{s,i} > u_{s,s}$; or $(i,j)$-(i) holds. In the latter case, we have that $M$ satisfies $R_{i,j}$ and $u_i$ reduces the distance of $z$ and we are done. In the former case, it follows that $(i,j)$-(ii) holds, so $M$ satisfies $R_{i,j}$. In this case, we have that $u_s$ reduces the distance of $z$. This concludes the proof of the result.
\end{proof}

\begin{theorem} \label{thm: Rij implies M dist red}
    Suppose that $M$ satisfies $R_{i,j}$ for all $2 \le i < j \le n$ then $M$ is distance reducing.
\end{theorem}

\begin{proof}
    We prove the result by induction on $n$. For the base case, take $n = 2$, and observe that the result holds trivially because $\ker(A)$ is generated by a single element. Fix $n > 2$. Assume that $R_{i,j}$ holds for all $2 \le i < j \le n$ and fix some $z \in \ker(A)$. We prove that there exists a move in $M$ that reduces the distance of $z$. We recall our convention that for each $i \in [n]$, the notation $(z)_i$ is the $i$th coordinate of $z$ and $z_i = |(z)_i|$ is its absolute value.

    Suppose $z_n = 0$. Let $A' = \begin{pmatrix}
        a_1 & \dots & a_{n-1}
    \end{pmatrix}$ and $M' = \{u'_2, \dots, u'_{n-1}\}$ where 
    \[
        u'_i = (u_{i,1}, u_{i,2}, \dots, u_{i,i-1}, -u_{i,i}, 0, \dots, 0) \in \ker(A')
    \]
    for each $i \in \{2, \dots, n-1\}$. By the gluing of $A$, we have that $M'$ is a Markov basis for $A'$. Observe that $M'$ satisfies $R_{i,j}$ for all $2 \le i < j \le n-1$, so by the inductive hypothesis, $M'$ is distance reducing for $\ker(A')$. In particular, $M'$ reduces the distance of the move
    \[
        z' = ((z)_1, \dots, (z)_{n-1}) \in \ker(A').
    \]
    Hence, there exists $u'_i \in M'$ such that $u'_i$ reduces the distance of $z'$. It follows immediately that $u_i \in M$ reduces the distance of $z$.

    So we may assume that $z_n \neq 0$. Without loss of generality, we may assume that $(z)_n < 0$. Since $M$ is a Markov basis and $z \in \ker(A)$, we may write 
    \[
        z = \sum_{i = 2}^n \lambda_i u_i
    \]
    for some integers $\lambda_i$. Since $(z)_n < 0$, it follows that $\lambda_n > 0$. Consider the set $N = \{i \in \{2, \dots, n-1\} : \lambda_i < 0\}$. We take cases on whether $N$ is the empty set.

    \medskip \noindent \textbf{Case 1.} Assume that $N \neq \emptyset$. Let $i = \max(N)$. So for all $i < j < n$, we have that $\lambda_j \ge 0$ and $\lambda_i < 0$. Now consider the $i$th coordinate of $z$:
    \[
        (z)_i = (-\lambda_i) u_{i,i} + \lambda_{i+1} u_{i+1,i} + \dots 
        + \lambda_{n-1} u_{n-1,i} + \lambda_{n} u_{n,i}. 
    \]
    Since $M$ satisfies $R_{i,n}$, we have that one of the conditions $(i,n)$-(i), $(i,n)$-(ii), or $(i,n)$-(iii) holds. If $(i,n)$-(i) holds then
    $u_{i,1} + \dots + u_{i,i-1} < u_{i,i} \le z_i$ and so $u_i$ reduces the distance of $z$. If $(i,n)$-(iii) holds then we have
    \[
        u_{n,1} + \dots + u_{n,i-1} + u_{n,i+1} + \dots + u_{n,n-1} < 
        u_{n,i} + u_{n,n}.
    \]
    Since $z_n \ge u_{n,n}$ and $z_{n,i} \ge u_{n,i}$, it follows immediately that $u_n$ is applicable to and reduces the distance of $z$. So it remains to consider the case when condition $(i,n)$-(ii) holds.

    Suppose that $(i,n)$-(ii) holds. Then there exists $j \in \{i+1, \dots, n-1 \}$ such that $u_j$ is a circuit
    \[
        u_j = (0, \dots, 0, u_{j,i}, 0, \dots, 0, -u_{j,j}, 0, \dots, 0)
    \]
    with $u_{j,i} > u_{j,j}$. Clearly if $u_{j,i} \ge z_i$, then $u_j$ is applicable to and reduces the distance of $z$. We now show that if $u_{j,i} < z_i$, then $z$ is distance reduced by $u_n$ or a circuit $u_{j'}$ for some $j' > j$. Assume that $u_{j,i} < z_i$. Since $\lambda_j \ge 0$, it follows that $\lambda_j = 0$. We mark this point in the proof with $(*)$.
    Next, we consider the $j$th coordinate of $z$:
    \[
    (z)_j = \lambda_{j+1} u_{j+1,j} + \dots + \lambda_n u_{n,j}.
    \]
    By assumption, $M$ satisfies $R_{j,n}$ so one of $(j,n)$-(i), $(j,n)$-(ii), or $(j,n)$-(iii) holds. Since $u_j$ is a circuit with $u_{j,i} > u_{j,j}$, it follows that $(j,n)$-(i) does not hold. If $(j,n)$-(iii) holds then it follows that $u_n$ reduces the distance of $z$. However, if $(j,n)$-(ii) holds then there exists $j' \in \{j+1, \dots, n-1 \}$ such that the move $u_{j'}$ is a circuit
    \[
        u_{j'} = (0, \dots, 0, u_{j', j}, 0, \dots, 0, -u_{j',j'}, 0, \dots, 0)
    \]
    where $u_{j',j} > u_{j',j'}$. If $u_{j',j} \ge z_j$, then $u_j$ is applicable and reduces the distance of $z$ and we are done. Otherwise, if $u_{j', i < z_j}$, then recall that $\lambda_{j'} \ge 0$, so we deduce that $\lambda_{j'} = 0$. We have now reached a point in the proof with exactly the same premises as $(*)$ except $i,j$ are replaced with the strictly larger values $j,j'$. Note that $j' < n$, so if we repeatedly apply this argument, then there are only finitely many steps until either $z$ is reduced by $u_n$ or a circuit $u_{j'}$. This concludes the proof for this case.

    \medskip \noindent \textbf{Case 2.} Assume that $N = \emptyset$. So we have $\lambda_i \ge 0$ for each $i \in \{2, \dots, n-1\}$. Consider the first coordinate $z_1$ of $z$ and the coefficient $\lambda_2 \ge 0$. If $\lambda_2 > 0$ then it follows that $u_2$ is applicable to $z$ because $u_{2,1} \le z_1$. By assumption $u_{2,1} > u_{2,2}$, so $u_2$ reduces the distance of $z$ and we are done. 
    
    We proceed to show that $z$ is either distance reduced by $u_n$ or a circuit $u_j$ for some $j < n$. So let us assume that $u_{2,1} > z_1$, and so $\lambda_2 = 0$. Therefore 
    $z_2 = \lambda_3 u_{3,2} + \dots + \lambda_n u_{n,2} \ge u_{n,2}$.
    By our original assumption, $M$ satisfies $R_{2,n}$ so one of $(2,n)$-(i), $(2,n)$-(ii), or $(2,n)$-(iii) holds. Since $u_{2,1} > u_{2,2}$, it follows that $(2,n)$-(i) does not hold. If $(2,n)$-(iii) holds then we have
    \[
        u_{n,1} + u_{n_3} + u_{n_4} + \dots + u_{n,n-1} < u_{n,2} + u_{n,n}.
    \]
    Since $u_{n,1} \le z_2$ and $u_{n,n} \le z_n$, it follows that $u_n$ reduces the distance of $z$. So it remains to consider the case when $(n,2)$-(ii) holds. In this case, there exists $i \in \{3, \dots, n-1\}$ such that $u_i$ is a circuit
    \[
        u_i = (0, u_{i,2}, 0, \dots, 0, -u_{i,i},0, \dots, 0)
    \]
    with $u_{i,2} > u_{i,i}$. If $u_{i,2} \le z_2$, then $u_i$ is applicable and reduces the distance of $z$. Otherwise if $u_{i,2} > z_2$, then it follows that $\lambda_i = 0$. We mark this point in the proof with $(*)$. In this case, we have that $M$ satisfies $R_{i,n}$ so one of $(i,n)$-(i), $(i,n)$-(ii), or $(i,n)$-(iii) holds. Since $u_i$ is a circuit with $u_{2,i} > u_{i,i}$, it follows that $(i,n)$-(i) does not hold. If $(i,n)$-(iii) holds then $u_n$ is applicable and reduces the distance of $z$. It remains to consider the case when $(i,n)$-(ii) holds. So there exists $j \in \{i+1, \dots, n-1\}$ such that $u_j$ is a circuit
    \[
        u_j = (0, \dots, 0, u_{j,i}, 0, \dots, 0, -u_{j,j}, 0, \dots, 0)
    \]
    with $u_{j,i} > u_{j,j}$. If $u_{j,i} \le z_i$, then $u_j$ is applicable and reduces the distance of $z$. Otherwise if $u_{j,i} > z_i$, then it follows that $\lambda_j = 0$. We have now reached a point a point in the proof with the same premises as $(*)$ except $2,i$ is replaced with the strictly larger values $i,j$. Note that $j < n$, so if we repeatedly apply the same argument, then there are only finitely many steps until $z$ is distance reduced by either $u_n$ or a circuit $u_j$. This concludes the proof of this case.
    
    \medskip

    So, in each case, we have shown that $z$ is distance reduced by some element $u_i \in M$. Since $z$ was arbitrary, $M$ is distance reducing.
\end{proof}

\section{Complete intersection monomial curves}\label{sec: ci mon curves}

In this section we prove that the distance reduction property governs the pattern of gluings for complete intersection monomial curves.

\begin{theorem}\label{thm: dist red ci implies first kind}
    Let $A \in \ZZ^{1 \times n}$ be a complete intersection.
    If a minimal Markov basis reduces the distance of the circuits of $A$, then $A$ admits a gluing of the first kind.
\end{theorem}

So, the distance reduction property for complete intersection monomial curves is characterised by the inequalities in Theorem~\ref{thm: characterisation dist red first kind} and Definition~\ref{def: first kind Rij conditions}. In particular, the circuits completely determine whether a minimal Markov basis is distance reducing.

\begin{corollary}\label{cor: ci dist red iff dist red circuits}
    Let $A \in \ZZ^{1 \times n}$ be a complete intersection and $M$ a minimal Markov basis for $A$. Then the following are equivalent:
    \begin{enumerate}
        \item $M$ is distance reducing,
        \item $M$ is reduces the distance of the circuits of $A$.
    \end{enumerate}
\end{corollary}

\begin{proof}
    By definition we have $(1) \implies (2)$. To prove $(2) \implies (1)$, by Theorem~\ref{thm: dist red ci implies first kind} we have that $A$ admits a gluing of the first kind. So by Theorem~\ref{thm: characterisation dist red first kind}, we have that $M$ is distance reducing for $A$.
\end{proof}

We develop tools and give a proof of Theorem~\ref{thm: dist red ci implies first kind} in the following sections. In Section~\ref{sec: condition for gluing of the first kind}, we show that gluings of the first kind may be detected with a combinatorial game. In Section~\ref{sec: distinguished circuit for w}, we set up the main notation for the proof and prove Lemma~\ref{lem: w cannot reduce circuits on one side}, which allows us to use an inductive argument in the proof of the theorem. In Section~\ref{sec: decorated gluing trees}, we introduce gluing trees and their decorations, which are combinatorial objects that track the sign patterns of the Markov basis elements. Section~\ref{sec: decorated gluing trees} concludes with a proof of Theorem~\ref{thm: dist red ci implies first kind}.

\subsection{A condition for gluings of the first kind} \label{sec: condition for gluing of the first kind}

In this section we give a combinatorial condition on the sign pattern of a Markov basis $M$ so that the matrix $A$ admits a gluing of the first kind. 

\begin{definition}
    Let $M = (m_{i,j})$ be a Markov basis for a matrix $A$. The \textit{sign matrix} of $M$ is the matrix $\sign(M) = (\sign(m_{i,j}))$ whose entries are from the set $\{-, 0, +\}$. For ease of notation we notation, we write $(\cdot)$ for the zero entries of the sign matrix.
\end{definition}

\noindent
\textbf{Sign Game.} Let $S \in \{-, 0, +\}^{k \times n}$ be a matrix of signs. A \textit{move} is given by selecting an entry $s_{i,j}$ and deleting the $i$th row and $j$th column of $S$. This move is \textit{valid} if:
\begin{itemize}
    \item the entry $s_{i,j}$ is the only nonzero entry of the $j$th column of $S$ and

    \item the entry $s_{i,j}$ is different from all other entries in the $i$th row of $S$.
\end{itemize}
We say $S$ is \textit{winnable} if there is a sequence of valid moves that deletes every element of $S$; resulting in the empty matrix.
If there is no sequence of valid moves that deletes every entry of $S$, then we say that $S$ is not winnable.

\begin{example}
    The following shows a sequence of valid moves in the sign game:
    \[
    \begin{bmatrix}
        +     & -     & \cdot & \cdot & \cdot & \cdot \\
        \cdot & \cdot & [+]   & -     & \cdot & \cdot \\
        +     & +     & \cdot & -     & \cdot & \cdot \\
        \cdot & -     & \cdot & \cdot & +     & -     \\
        \cdot & +     & \cdot & \cdot & -     & -     \\
    \end{bmatrix},
    \begin{bmatrix}
        +     & -     & \cdot & \cdot & \cdot \\
        +     & +     & [-]   & \cdot & \cdot \\
        \cdot & -     & \cdot & +     & -     \\
        \cdot & +     & \cdot & -     & -     \\
    \end{bmatrix},
    \begin{bmatrix}
        [+]   & -     & \cdot & \cdot \\
        \cdot & -     & +     & -     \\
        \cdot & +     & -     & -     \\
    \end{bmatrix},
    \begin{bmatrix}
        -     & +     & -     \\
        +     & -     & -     \\
    \end{bmatrix}.
    \]
    Each move in given by removing the row and column of the bracketed entry. The final matrix in the sequence admits no valid moves so none of these matrices is winnable.
\end{example}

\begin{proposition}\label{prop: first kind iff winnable}
    Let $A$ be a complete intersection. Let $M$ be a minimal Markov basis for $A$. Then $A$ admits a gluing of the first kind if and only if $\sign(M)$ is winnable.
\end{proposition}

\begin{proof}
    The matrix $A$ admits a gluing of the first kind if and only if, up a permutation of the rows and columns of $M$ and sign of the rows of $M$, the matrix $\sign(M) = (s_{i,j})$ is given by
    \begin{equation} \label{eqn: sign matrix of first kind}
    \sign(M) = \begin{bmatrix}
        +      & -      & \cdot  & \cdot  & \cdots & \cdot \\
        \oplus & \oplus & -      & \cdot  & \cdots & \cdot \\
        \oplus & \oplus & \oplus & -      & \cdots & \cdot \\
        \vdots & \vdots & \vdots & \ddots & \ddots & \vdots \\
        \oplus & \oplus & \oplus & \cdots & \oplus & - 
    \end{bmatrix}.
    \end{equation}
    The matrix $\sign(M)$ is winnable since the sequence of moves: $s_{n-1,n}, s_{n-2, n-1}, \dots, s_{2,3}, s_{1,2}$ deletes the entire matrix. Conversely, every winnable $(n-1) \times n$ matrix $S = (s_{i,j})$ can be transformed into the above matrix with a permutation of the rows and columns of $S$ and negating some of the rows. To see this, suppose that a winning sequence of moves is given by $s_{i_1,j_1}, s_{i_2, j_2}, \dots, s_{i_{n-1}, j_{n-1}}$. Consider the first move $s_{i_1, j_1}$. Without loss of generality, we may assume that $s_{i_1, j_1}$ is $(-)$ because we may change the sign of the row. We move that entry to the bottom-right corner of the matrix to form a new matrix $M'$ with sign matrix $\sign(M') = (s'_{i,j})$. By the definition of a winnable matrix, the entries above $s'_{n-1, n} = s_{i_1, j_1}$ in $\sign(M')$ are all zero and the entries to the left are all $(+)$ or $(\cdot)$. We continue in this way for all subsequent moves and results in matrix as (\ref{eqn: sign matrix of first kind}). Therefore $A$ admits a gluing of the first kind.
\end{proof}

To prove Theorem~\ref{thm: dist red ci implies first kind}, the first step is to apply Lemma~\ref{lem: w cannot reduce circuits on one side}, which shows that there exists a distinguished circuit $z$ of $A$ such that the only element of $M$ that reduces the distance of $z$ is $w$. In particular, since $w$ is applicable to this circuit, we deduce that many of the entries of $w$ are zero. We proceed to examine the circuits not distance reduced by $w$ to find further zeros in the coordinates of the Markov basis. We keep track of the zeros to show that $\sign(M)$ is winnable.

\begin{example}
    Let $A \in \ZZ^{1 \times 5}$ be a matrix. Suppose that $A$ is a complete intersection that admits the gluing 
    \[
        A = (((a_1 \circ a_2) \circ a_3) \circ (a_4 \circ a_5)).
    \]
    Without loss of generality we assume $a_1 < a_2$ and $a_4 < a_5$. Suppose that $M$ is a minimal Markov basis of the form
    \[
        M = \{u_2, u_3, u_4, u_5\} =  \begin{bmatrix}
            u_{2,1} & -u_{2,2} & 0 & 0 & 0 \\
            u_{3,1} & u_{3,2} & -u_{3,3} & 0 & 0 \\
            0 & 0 & 0 & u_{4,4} & - u_{4,5} \\
            u_{5,1} & u_{5,2} & u_{5,3} & -u_{5,4} & -u_{5,5}
        \end{bmatrix}
    \]
    where $u_{i,j} \ge 0$ for all $i$, $j$. By Lemma~\ref{lem: w cannot reduce circuits on one side}, we have that $u_5$ reduces the distance of one of the circuits $z_{2,5}$ or $z_{3,5}$ where $z_{i,j}$ is the circuit supported on $\{i, j\}$. By considering the applicability of $u_5$ to these circuits, we identify three cases for $u_5$: $\supp(u_5^-) = \{5\}$, $\supp(u_5^+) = \{2\}$, and $\supp(u_5^+) = \{3\}$.

    \medskip \noindent \textbf{Case 1.} $u_5 = (u_{5,1}, u_{5,2}, u_{5,3}, 0, -u_{5,5})$. In this case we have
    \[
        \sign(M) = \begin{bmatrix}
            + & - & \cdot & \cdot & \cdot \\
            \oplus & \oplus & - & \cdot & \cdot \\
            \cdot & \cdot & \cdot & + & - \\
            \oplus & \oplus & \oplus & \cdot & - 
        \end{bmatrix},
    \]
    which is winnable and so $A$ admits a gluing of the first kind. For instance, a winning sequence of moves is given by the indices: $(3, 4), (4, 5), (2, 3), (1, 2)$.

    \medskip \noindent \textbf{Case 2.} $u_5 =
    (0, u_{5,2}, 0, -u_{5,4}, -u_{5,5})$. In this case we have
    \[
        \sign(M) = \begin{bmatrix}
            + & - & \cdot & \cdot & \cdot \\
            \oplus & \oplus & - & \cdot & \cdot \\
            \cdot & \cdot & \cdot & + & - \\
            \cdot & + & \cdot & \ominus & \ominus 
        \end{bmatrix},
    \]
    which is winnable and so $A$ admits a gluing of the first kind. For instance, a winning sequence of moves is given by the indices: $(2, 3), (1, 1), (4, 2), (3, 4)$.

    \medskip \noindent \textbf{Case 3.} $u_5 =
    (0, 0, u_{5,3}, -u_{5,4}, -u_{5,5})$. In this case we have that $u_5$ reduces the distance of the circuit $z_{3,5} = (0, 0, z_3, 0, -z_5)$ and $u_5$ is applicable to the positive part of $z$. Note that if $u_5$ also reduces the distance of $z_{2,5}$ then by the above cases, we have that $A$ admits a gluing of the first kind. So, let us assume that $u_5$ does not reduce the distance of $z_{2,5}$. So we must have that $u_2$ reduces the distance of $z_{2,5}$. In particular, we have that $u_2$ is applicable to $z_{2,5}$ and so $u_2 = (0, u_{2,2}, -u_{2,3}, 0, 0)$.
    Therefore the sign of $M$ is given by
    \[
        \sign(M) = \begin{bmatrix}
            + & - & \cdot & \cdot & \cdot \\
            \cdot & + & - & \cdot & \cdot \\
            \cdot & \cdot & \cdot & + & - \\
            \cdot & \cdot & + & \ominus & \ominus 
        \end{bmatrix},
    \]
    which is winnable and so $A$ admits a gluing of the first kind. In this case, a winning sequence of moves is given by the indices: $(1,1), (2,2), (4,3), (3,4)$.
\end{example}

\subsection{Existence of distinguished circuits}\label{sec: distinguished circuit for w}

In this section we set up the notation for the proof of Theorem~\ref{thm: dist red ci implies first kind}. In Lemma~\ref{lem: w cannot reduce circuits on one side} we prove that key result that allows us to apply induction.

\medskip

\noindent \textbf{Setup.} Throughout this section and the next, we consider a pair of matrices $A = (a_1, \dots, a_n)$ and $B = (b_1, \dots, b_m)$ with $n, m \ge 1$ that are complete intersections. We assume there is a gluing $C = A \circ B$. Note that all complete intersection monomial curves arise in this way. We assume that $M$ is a Markov basis for $C$ that has the form
$M = \{u_1, \dots, u_{n-1}, v_1, \dots, v_{m-1}, w \}$ where:
\begin{itemize}
    \item for each $i \in [n-1]$ we have $\supp(u_i) \subseteq \{1,\dots,n\}$,
    \item the set of projections of $u_i$ onto the coordinates $\{1, \dots, n\}$ is a Markov basis for $A$,
    \item for each $j \in [m-1]$ we have $\supp(v_j) \subseteq \{n+1, \dots, n+m \}$,
    \item the set of projections of $v_j$ onto the coordinates $\{n+1, \dots, n+m \}$ is a Markov basis for $B$,
    \item $\supp(w^+) \subseteq \{1, \dots, n \}$ and $\supp(w^-) \subseteq \{n+1, \dots, n+m\}$.
\end{itemize}
We write $M_A = \{u_1, \dots, u_{n-1} \}$ and $M_B = \{v_1, \dots, v_{m-1} \}$.
For each pair of indices $i, j \in \{1, 2, \dots, n+m \}$ with $i < j$, we denote by $z_{i,j}$ the circuit supported on $i$ and $j$ given by
\[
    z_{i,j} = (0, \dots, 0, (z_{i,j})_i, 0, \dots, 0, -(z_{i,j})_j, 0, \dots, 0)
\]
for some positive integers $(z_{i,j})_i$ and $(z_{i,j})_j$. For ease of notation, we define $z_{j,i}$ to be equal to $z_{i,j}$.

Our proof of Theorem~\ref{thm: dist red ci implies first kind} is by induction on $n+m$. To apply the inductive hypothesis, we must show that either: the set $M_B \cup \{w\}$ does not reduce the distance of any circuit $z_{i,j}$ with $i, j \in [n]$, or the set $M_A \cup \{w\}$ does not reduce the distance of any circuit $z_{n+i, n+j}$ with $i,j \in [m]$. We may then apply induction to either $A$ with Markov basis given by a projection of $M_A$ or $B$ with Markov basis given by a projection $M_B$, respectively. In the following lemmas, we prove that one of these cases indeed holds. We now show the existence of distinguished circuits of $C$ that are distance reduced by a single element of $M$.

\begin{lemma}\label{lem: non reducible coordinate}
    Suppose $M$ reduces the distance of the circuits of $C$. Then at least one of the following holds:
    \begin{enumerate}
        \item There exists $r \in [n]$ such that for each $j \in [m]$ the circuit $z_{r, n+j}$ is distance reduced by exactly one element $u_j \in M$. Moreover, this unique element lies in the set $M_B \cup \{w\}$.
        In particular, there is a bijection $u_j \leftrightarrow z_{r, n+j}$ between the set $M_B \cup \{w\}$ and the circuits $\{z_{r, n+j} : j \in [m]\}$ where $u_j$ is the only element of $M$ to distance reduce the circuit $z_{r, n+j}$.

        \item There exists $c \in [m]$ such that for each $i \in [n]$ the circuit $z_{i, n+c}$ is distance reduced by exactly one element $u_i \in M$. Moreover, this unique element lies in the set $M_A \cup \{w\}$. In particular, there is a bijection $u_i \leftrightarrow z_{i, n+c}$ between the set $M_A \cup \{w\}$ and the circuits $\{z_{i, n+c} : i \in [n]\}$ where $u_i$ is the only element of $M$ to distance reduce the circuit $z_{i, n+c}$.
    \end{enumerate}
\end{lemma}

\begin{proof}
    Let $T$ be a rectangular grid with $n$ rows and $m$ columns. We fill the $(i,j)$ entry of $T$ with the set
    \[
        T_{i,j} = \{x \in M : x \text{ reduces the distance of } z_{i, n+j}\}.
    \]
    
    \medskip \noindent \textbf{Claim.} Each element $u \in M_A$ appears in at most one row of $T$.

    \begin{proof}
        Suppose that an element $u \in M_A$ reduces the distance of some circuit $z_{i,n+j}$ for some $i \in [n]$ and $j \in [m]$. Since $\supp(u) \subseteq [n]$ and $u$ is applicable to the circuit $z_{i,n+j}$, it follows that $\supp(u^+) = \{i\}$ or $\supp(u^-) = \{i\}$. Since we may freely replace $u$ with $-u$ in $M$, we may assume without loss of generality that $\supp(u^+) = \{i\}$. Since $u$ reduces the distance of $z_{i,n+j}$, we have
        $|z_{i,n+j}| > |z_{i,n+j} - u|$ so 
        \[
            (z_{i,n+j})_i + (z_{i,n+j})_j = |z_{i,n+j}| > |z_{i,n+j} - u| = 
            (z_{i,n+j})_i - (u)_i + |u^-| + (z_{i,n+j})_j.
        \]
        Hence $(u)_i > |u^-|$. Assume by contradiction that $u$ reduces the distance of another circuit $z_{i', n+j'}$ with $i' \in [n]$, $i' \neq i$, and $j' \in [m]$. By a similar argument to the above, it follows that $\supp(u^-) = \{i'\}$ and $(u)_{i'} > |u^+|$. However, we have $|u^+| = (u)_i > |u^-| = (u)_{i'} > |u^+|$, which is a contradiction. This finishes the proof of the claim.
    \end{proof}

    Since $|M_A| = n-1$, it follows that there is a row of $T$ that does not contain any element of $M_A$. Let $r \in [n]$ denote the index of this row, i.e., for each $j \in [m]$ we have $T_{r, j} \cap M_A = \emptyset$.

    The above claim is symmetrical by switching $M_A$ with $M_B$ and switching rows of $T$ with columns of $T$. So, each element $v \in M_B$ appears in at most one column of $T$. Since $|M_B| = m-1$, it follows that there is a column of $T$ that does not contain any element of $M_B$. Let $c \in [m]$ denote the index of this column.

    Since $M$ reduces the distance of the circuits of $C$, there is an element of $M$ that reduces the distance of the circuit $z_{r,n+c}$. However, by the above, all elements of $M_A \cup M_B$ do not reduce the distance of $z_{r, n+c}$, so it follows  that $w$ is the unique element of $M$ that reduces the distance of $z_{r, n+c}$.

    To finish the proof of the result, we show that the entries of $T$ containing $w$ are contained in a single row or column. Assume not. Then $w$ reduces the distance of a pair of circuits $z_{i, n+j}$ and $z_{i', n+j'}$ with $i, i' \in [n]$, $j, j' \in [m]$, $i \neq i'$, and $j \neq j'$. Since $w$ is applicable to the circuits $z_{i, n+j}$ and $z_{i', n+j'}$, it straightforward to show that $w$ is a circuit and $\supp(w) \subseteq \{i, i', j, j'\}$. Moreover there are two cases: either $w$ is supported on $i$ and $n+j'$, or $w$ is supported on $i'$ and $n+j$. These cases are symmetrical so we may assume that $w$ is supported on $i$ and $n+j'$, i.e., $w = (0, \dots, 0, w_i, 0, \dots, 0, -w_{n+j'}, 0, \dots, 0)$. Since $w$ reduces the distance of $z_{i,n+j}$ it follows that $w_i > w_{n+j'}$. Since $w$ reduces the distance of $z_{i', n+j'}$ it follows that $w_{n+j'} > w_i > w_{n+j'}$, which is a contradiction. This concludes the proof.
\end{proof}

\begin{lemma}\label{lem: w cannot reduce circuits on one side}
    If $M$ reduces the distance of the circuits of $C$, then one of the following holds:
    \begin{enumerate}
        \item[(a)] $|\supp(w^-)| = 1$, condition (1) from Lemma~\ref{lem: non reducible coordinate}, and for each $i, j \in [n]$ the move $w$ does not reduce the distance of the circuit $z_{i,j}$,

        \item[(b)] $|\supp(w^+)| = 1$, condition (2) from Lemma~\ref{lem: non reducible coordinate}, and for each $i, j \in [m]$ the move $w$ does not reduce the distance of the circuit $z_{n+i, n+j}$.
    \end{enumerate}
\end{lemma}

\begin{proof}
    By Lemma~\ref{lem: non reducible coordinate} we have that $w$ is the unique element of $M$ that reduces the circuit $z_{r, n+c}$ for some $r \in [n]$ and $c \in [m]$. In particular, $w$ is applicable to $z_{r,n+c}$ so either $\supp(w^+) = \{r\}$ or $\supp(w^-) = \{n+c\}$. We take cases based on the support of $w$.

    \medskip \noindent \textbf{Case 1.} Let $\supp(w^+) = \{r\}$ and $\supp(w^-) = \{c\}$. We take further cases on $a_r$ and  $b_c$.

    If $a_r = b_c$, then the only circuit $z_{i,n+j}$ with $i \in [n]$ and $j \in [m]$ that is reduced by $w$ is exactly $z_{r, n+c}$. Therefore both conditions (1) and (2) hold in Lemma~\ref{lem: non reducible coordinate}. In this case we observe that $w$ does not reduce the distance of any other circuit of $C$. Hence both conditions (a) and (b) hold.

    If $a_r > b_c$, then $w$ reduces the distance of all circuits $z_{r, n+j}$ with $j \in [m]$ so condition (1) of Lemma~\ref{lem: non reducible coordinate} does not hold, hence condition (2) holds. Moreover, the move $w$ does not reduce the distance of any circuit $z_{n+i,n+j}$ with $i, j \in [m]$, hence (b) holds.

    If $a_r < b_c$, then $w$ reduces the distance of all circuits $z_{i, n+c}$ with $i \in [n]$ so condition (2) of Lemma~\ref{lem: non reducible coordinate} does not hold, hence condition (1) holds. Since $w$ does not reduce the distance of any circuit $z_{i,j}$ with $i, j \in [n]$, we have that (a) holds. 

    \medskip \noindent \textbf{Case 2.} Let $\supp(w^+) = \{r\}$ and $\supp(w^-) \supsetneq \{c\}$. In this case $w$ is not applicable to any circuit $z_{n+i,n+j}$ with $i,j \in [m]$. In particular, the move $w$ does not reduce the distance of these circuits. To show that (b) holds, it suffices to prove that condition (2) from Lemma~\ref{lem: non reducible coordinate} holds. If not then $w$ reduces the distance of a circuit $z_{i, n+c}$ for some $i \in [n]$ with $i \neq r$. However this is an immediate contradiction because $w$ is not applicable to $z_{i, n+c}$ since $i \neq r$ and $|\supp(w^-)| \ge 2$. So we have shown (b) holds.

    \medskip \noindent \textbf{Case 3.} Let $\supp(w^-) = \{c\}$ and $\supp(w^+) \supsetneq \{r\}$. In this case $w$ is not applicable to any circuit $z_{i,j}$ with $i,j \in [n]$. In particular, the move $w$ does not reduce the distance of these circuits. To show that (a) holds, it suffices to prove that condition (1) from Lemma~\ref{lem: non reducible coordinate} holds. If not then $w$ reduces the distance of a circuit $z_{r, n+j}$ for some $j \in [m]$ with $j \neq c$. However, this is an immediate contradiction because $w$ is not applicable to $z_{r, n+j}$ since $j \neq c$ and $|\supp(w^+)| \ge 2$. So we have shown (a) holds.
    
    \medskip \noindent
    In each case we have shown that either (a) or (b) holds. This concludes the proof.
\end{proof}

\subsection{Gluing trees} \label{sec: decorated gluing trees}
Recall that complete intersection monomial curves are completely glued. This means any such curve $C = A \circ B \in \NN^{1 \times (n+m)}$ is the gluing of a complete intersection monomial curves $A \in \NN^{1 \times 1}$ and $B \in \NN^{1 \times m}$. Similarly, both $A$ and $B$ admit gluings and so on until each constituent matrix has size one. We record the data of the sequence of gluing with a rooted binary tree $\MT_C$. The leaves of $\MT_C$ are labelled with the entries of $C$, which we identify with the set $[n+m]$. The graph structure is defined inductively. If $n = 1$ and $A = (a_1)$, then $\MT_A$ is the graph with a single vertex $a_1$, called the root of $\MT_A$, and no edges. Suppose $\MT_A$ and $\MT_B$ are the rooted binary trees associated to $A$ and $B$ with roots $u$ and $v$ respectively. Then $\MT_C$ is defined to be the graph obtained from the disjoint union $\MT_A \sqcup \MT_B$ by adding a vertex $w$ adjacent to $u$ and $v$. The vertex $w$ is defined to be the root of $\MT_C$. We call $\MT_C$ the \textit{gluing tree} of $C$. 

\medskip
\noindent \textbf{Gluing tree notation.} There is a canonical embedding of $\MT_C$ in $\RR^2$. The embedding is defined so that: the vertex $i \in [n+m]$ lies at position $(i,0) \in \RR^2$; each edge is a line segment with gradient $1$ or $-1$; and each parent has a higher $y$-coordinate than its children. For example, see the gluing tree in Figure~\ref{fig: gluing tree example}. Thus, given a non-leaf vertex $v$ of $\MT_C$, there are two edges incident to $v$ below it: one to the \textit{left} and one to the \textit{right}. The set of leaves of $\MT_C$ connected to $v$ by a path whose first edge is the left one is denoted $L(v)$. Similarly, the set of leaves of $\MT_C$ connected to $v$ be a path whose first edge is the right one is denoted $R(v)$. We call the set of vertices $L(v) \cup R(v)$ the leaves of $\MT_C$ \textit{below} $v$.

Given a Markov basis $M$ for $C$, there is a natural bijection between the elements of $M$ and non-leaf vertices of $\MT_C$. This bijection associates an element $u \in M$ with a non-leaf vertex $v$ so that one of the following holds:
\[
\left(\supp(u^+) \subseteq L(v) \text{ and } \supp(u^-) \subseteq R(v) \right)
\text{ or }
\left( \supp(u^-) \subseteq L(v) \text{ and } \supp(u^+) \subseteq R(v) \right).
\]

\begin{example}\label{example: gluing tree in A5}
    Let $A = \begin{pmatrix}
        a_1 & a_2 & a_3
    \end{pmatrix}$ and $B = \begin{pmatrix}
        b_1 & b_2
    \end{pmatrix}$. Suppose there is a gluing $C = A \circ B$. Then the gluing tree $\MT_C$ is shown in Figure~\ref{fig: gluing tree example}. Given a Markov basis $M$, the element $w$ is the root of $\MT_C$ and corresponds to the final row in the sign matrix in the diagram. The non-leaf vertices of the gluing tree are labelled with the corresponding elements of the Markov basis. 

    \begin{figure}[t]
        \centering
        \includegraphics{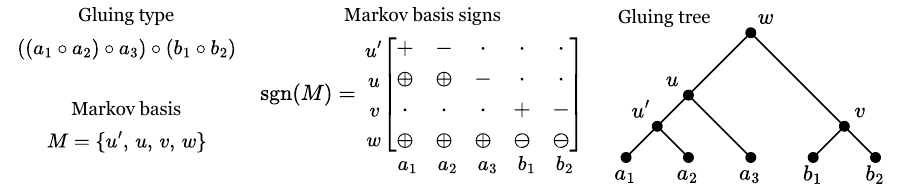}
        \caption{
        The gluing tree in $\mathbb A^5$ from Example~\ref{example: gluing tree in A5}}
        \label{fig: gluing tree example}
    \end{figure}
\end{example}

\noindent \textbf{Decorated partial gluing trees.}
Let $M$ be a Markov basis for $C$ and suppose that $M$ reduces the distance of the circuits of $C$. By swapping $A$ and $B$ in the gluing of $C$, we may assume that condition (b) of Lemma~\ref{lem: w cannot reduce circuits on one side} and \ref{lem: non reducible coordinate} holds. The \textit{partial gluing tree} $\MT_C(M)$ is the induced subgraph of $\MT_C$ obtained by deleting the leaves $b_1, \dots, b_m$ and vertices $v$ that correspond to elements of $M_B$. Explicitly, these are the non-leaf vertices $v$ of $\MT_C$ such that $L(v) \cup R(v) \subseteq \{b_1, \dots, b_m\}$. So, the partial gluing tree $\MT_C(M)$ has $n$ leaves and $n$ non-leaves. Its leaves are labelled with the elements $[n]$ corresponding to $A$ and its non-leaves are labelled with the elements of $M_A \cup \{w\}$.

By condition (b) of Lemma~\ref{lem: non reducible coordinate}, there exists $c \in [m]$ and a bijection between $[n]$ and $M_A \cup \{w\}$ given by $i \mapsto u_i$, with the property that $u_i$ is the only element of $M$ that reduces the distance of the circuit $z_{i, n+c}$. Fix $u \in M_A \cup \{w\}$ and let $i \in [n]$ be the corresponding element in $[n]$ as above. We define the \textit{decoration $D_u$} to be the following pair of subgraphs $D_u = (s_u, d_u)$ of $\MT_C(M)$. The subgraph $s_u$ is the unique path from the non-leaf vertex $u$ to the leaf $i$. Let $\ell$ be edge incident to $u$ along this path. The subgraph $d_u$ is the induced subgraph of $\MT_C(M)$ consisting of all vertices below $u$ whose path to $u$ does not involve $\ell$. Note that $d_w$ is the empty graph.

We define the \textit{decorated partial gluing tree} $\overline{\MT_C(M)}$ to be the pair $(\MT_C(M), \MD)$, where $\MD = \{D_u : u \in M_A \cup \{w\} \}$ is the set of decorations. We note that the operation of taking a graph minor of $\MT_C(M)$ naturally extends to any subgraph, hence taking graph minors naturally extends to an operation on $\overline{\MT_C(M)}$.

\begin{example}\label{example: decorated partial gluing tree in A5}
    Consider the gluing tree in Example~\ref{example: gluing tree in A5}. Suppose that condition (b) holds for some Markov basis $M$ of $C$. Let us also assume that the bijection between $\{1,2,3\}$ and $M_A \cup \{w\} = \{u', u, w\}$ is given by: $w \leftrightarrow 2$, $u \leftrightarrow 3$, and $u' \leftrightarrow 1$. The decorated partial gluing tree $\overline{\MT_C(M)}$ is shown in Figure~\ref{fig: decorated gluing tree}.

    \begin{figure}[t]
        \centering
        \includegraphics{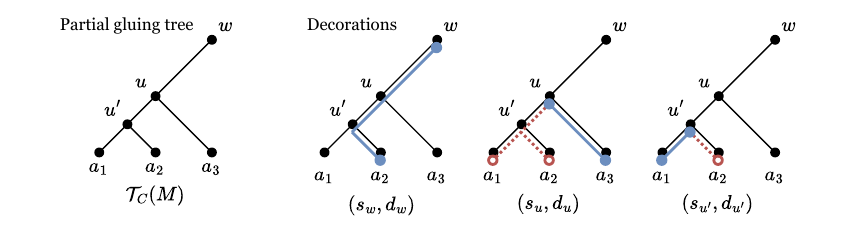}
        \caption{The decorated partial gluing tree in Example~\ref{example: decorated partial gluing tree in A5}. In the figures of decorated trees, the solid line shows the path $s_v$ and the dotted lines indicate $d_v$, for each non-leaf $v$.}
        \label{fig: decorated gluing tree}
    \end{figure}

    Each decoration encodes information about the sign patterns in the Markov basis. For example, the path $s_w$ indicates that $\supp(w^+) = \{2\}$ and $\supp(w^-) \subseteq \{n+1, \dots, n+m\}$, which follows from Lemma~\ref{lem: w cannot reduce circuits on one side}. Similarly, the subgraph $d_u$ encodes the fact that $\supp(u^-) \subseteq \{1, 2\}$.
\end{example}

\begin{definition}
    An \textit{abstract decorated partial gluing tree} (or \textit{abstract d-tree}) is the combinatorial data of a decorated partial gluing tree. Explicitly, an abstract d-tree is a pair $(\MT, \MD)$ where $\MT$ is rooted tree such that: the root has degree $1$ and every non-root non-leaf vertex has degree $3$. Following the notation for gluing trees, we assume that $\MT$ is embedded in $\RR^2$. We write $V(\MT)$ for the set of non-leaf vertices of $\MT$, which includes the root, and $\ell(\MT)$ for the set of leaves of $\MT$. The set $\MD$ is the set of decorations $D_u$, with one for each $u \in V(\MT)$. A decoration $D_u = (s_u, d_u)$ is a pair of subgraphs $s_u$ and $d_u$. The subgraph $s_u$ is a path in $\MT$ connecting $u$ to a leaf $\ell_u$. Let $\ell$ be the first edge of this path. The subgraph $d_u$ is the induced subgraph of $\MT$ consisting of all vertices $v$ below $u$ such that the path from $v$ to $u$ avoids $\ell$. In addition, we require that the set $\{\ell_u : u \in V(\MT)\} = \ell(\MT)$.
\end{definition}

\begin{example}
    Figure~\ref{fig: decoration of abstract d-tree} shows the decoration of an abstract d-tree.

    \begin{figure}[t]
        \centering
        \includegraphics{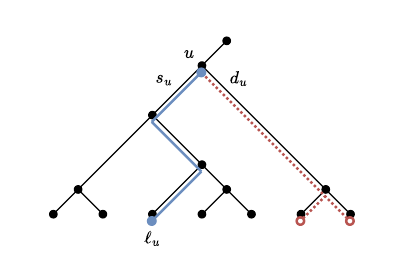}
        \caption{A decoration of an abstract d-tree}
        \label{fig: decoration of abstract d-tree}
    \end{figure}
\end{example}

\begin{lemma}\label{lem: abstract d-tree special leaf}
    Let $(\MT, \MD)$ be an abstract d-tree. Then there exists a non-leaf vertex $u \in V(\MT)$ such that $\ell_u$ and $u$ are adjacent and $\ell_u$ is not contained in the subgraph $d_v$ for each $v \in V(\MT)$.
\end{lemma}

\begin{proof}
    We prove the result for all abstract d-trees $(\MT, \MD)$ by induction on $n$ the number of leaves. If $n = 1$, then the graph $\MT$ is the graph with a root $w$ and a single leaf $\ell$. The path $s_w$ is the edge that connects $w$ and $\ell$ and $d_w$ is the empty graph. Thus, the result holds.

    Suppose that $n \ge 2$. Then there exists a non-leaf vertex $v$ that is adjacent to two leaves: $\ell_v$ and $\ell$. If $\ell_v$ is not contained in $d_u$ for all $u$ then we are done. Otherwise assume that $\ell_v$ is contained in $d_u$ for some $u \in V(\MT)$. We construct a new abstract d-tree by \textit{pruning} the vertex $v$ and the leaf $\ell_v$, defined by the following diagram.

    \begin{center}
    \includegraphics[]{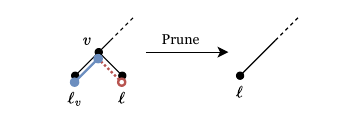}
    \end{center}

    This operation produces an abstract d-tree with one fewer leaf. By induction, this abstract d-tree has a vertex $v'$ that is adjacent to leaf $\ell'$ such that $\ell'$ does not belonging to $d_u$  for all $u$. Observe that $v'$ and $\ell'$ are adjacent in the original abstract d-tree and $\ell'$ does not belong to $d_u$ for all $u$ in original abstract d-tree. So we are done.
\end{proof}

\begin{lemma}\label{lem: sequence of special vertices}
    Let $(\MT, \MD)$ be an abstract d-tree with $n$ leaves. There exists an ordering $(u_1, u_2, \dots, u_n)$ of $V(\MT)$ such that, for each $i \in [n]$, the leaf $\ell_{u_i}$ is not contained in $d_{u_j}$ for all $j \ge i$.
\end{lemma}

\begin{proof}
    We proceed by induction on $n$. For $n = 1$ the result is straighforward. Assume $n \ge 2$. By Lemma~\ref{lem: abstract d-tree special leaf}, there is a leaf $\ell_1$ that does not belong to $d_u$ for all $u \in V(\MT)$. In particular the non-leaf vertex $u_1 \in V(\MT)$ such that $\ell_{u_1} = \ell_1$ is adjacent to $\ell_1$ so we may \textit{prune} $(\MT, \MD)$ as follows.

    \begin{center}
        \includegraphics[]{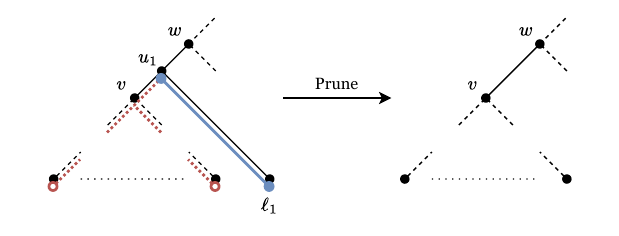}
    \end{center}

    The resulting abstract d-tree has one fewer leaf so by induction there exists a sequence $(u_2, \dots, u_n)$ such that $\ell_{u_i}$ is not contained in $d_{u_j}$ for all $j \ge i$. So the sequence $(u_1, u_2, \dots, u_n)$ satisfies the conditions of the result and we are done.
\end{proof}

We now use this description together with the decorated partial gluing tree $\overline{\MT_C(M)}$ to prove our main result.

\begin{proof}[Proof of Theorem~\ref{thm: dist red ci implies first kind}]
    We prove that if $M$ is a distance reducing Markov basis for the matrix $C = A \circ B$, then $C$ admits a gluing of the first kind. The proof is by induction on $m+n$. If $m+n = 2$, then the result is clear. 
    
    Assume that $m+n \ge 3$.
    First we apply Lemma~\ref{lem: w cannot reduce circuits on one side} where we may assume that (b) holds, since if (a) holds, then we may swap $A$ and $B$. In particular, $w$ does not reduce the distance of the circuits supported on $\{n+1, \dots, n+m \}$. Also observe that none of the moves in $M_A$ reduce the distance of these circuits as they are not applicable. So, we may apply the induction hypothesis to the matrix $B$ and Markov basis for $B$ obtained by projecting $M_B$ to the last $m$ coordinates. So the matrix $\sign(M_B)$ is winnable, so $\sign(M_B)$ has a winning sequence that we denote by $s_B$.

    We now extend $s_B$ to a winning sequence for $M$.
    Consider the decorated partial gluing tree $\overline{\MT_C(M)}$. By Lemma~\ref{lem: sequence of special vertices}, there is a sequence of vertices $(u_1, \dots, u_n)$ such that for each $i \in [n]$ the leaf $\ell_{u_i}$ is not contained in $d_{u_j}$ for all $j \in [n]$ with $j \ge i$. We now interpret this sequence in terms of the sign matrix of $M$. By the construction of $\overline{\MT_C(M)}$, for each $i \in [n]$ the move $u_i$ is the unique element of $M$ that reduces the distance of the the circuit $z$ that is supported on $\{\ell_{u_i}, n+c \}$ for some $c \in [m]$. In particular, $u_i$ is applicable $z$ so it follows that $\supp(u_i^+) = \{\ell_{u_i}\}$ or $\supp(u_i^-) = \{\ell_{u_i}\}$. Without loss of generality assume $\supp(u_i^+) = \{\ell_{u_i}\}$ and $\ell_{u_i} \in L(u_i)$, otherwise we may consider $-u_i$ or swap the two subgraphs below $u_i$ in drawing of $\MT_C(M)$. By construction, we have that $\supp(u_i^-) \subseteq R(u_i) = \ell(d_{u_i})$.

    Since $\ell_{u_1}$ is not contained in any subgraph $d_v$ for all $v \in V(\MT_C(M))$, it follows that the column of $M$ indexed by $\ell_{u_1}$ has one nonzero entry in the row indexed by $u_1$. By the assumptions above, the sign of this entries is $-$ and is unique in the row. Therefore $(u_1, \ell_{u_1})$ gives a valid move for the matrix $M$. After removing row $u_i$ and column $\ell_{u_i}$ from $M$, we similarly follows that $(u_2, \ell_{u_2})$ is a valid move. We continue in this way until we are left with the matrix $\sign(M_B)$ restricted to the last $m$ columns.
    We may then apply the moves $s_B$ to delete every entry from the matrix. Thus we have shown that the matrix $M$ is winnable by the sequence of moves
    \[
        (u_1, \ell_{u_1}), \, 
        (u_2, \ell_{u_2}), \, 
        \dots ,\, 
        (u_n, \ell_{u_n}), \, 
        s_B.
    \]
    Therefore, by Proposition~\ref{prop: first kind iff winnable}, we have $C$ admits a gluing of the first kind. 
\end{proof}

\section{Monomial curves in \texorpdfstring{$\mathbb A^4$}{A4}}\label{sec: dim4 mon curves}

In this section we explore monomial curves in $\mathbb A^4$. Let $A = \begin{pmatrix}
    a_1 & a_2 & a_3 & a_4
\end{pmatrix}$ be a matrix of positive integers. Let $M$ be a minimal Markov basis for $A$. If $A$ is a complete intersection, then by Corollary~\ref{cor: ci dist red iff dist red circuits} we have that $M$ is distance reducing if and only if $M$ reduces the distance of the circuits of $A$. The next example shows that if $A$ is not a complete intersection then the same result does not necessarily hold.

\begin{example}\label{example: dim4 reduces circuits but not reducing}
    Consider the matrix $A = \begin{pmatrix}
        14 & 21 & 23 & 29
    \end{pmatrix}$ and the minimal Markov basis
    \[
    M = \begin{bmatrix}
        1 & 1 & 1 & -2 \\
        3 & -2 & 0 & 0 \\
        3 & 1 & -4 & 1 \\
        7 & 0 & -3 & -1
    \end{bmatrix}.
    \]
    It is not difficult to show that $M$ distance reduces the circuits of $A$. However, $M$ cannot reduce the distance of the element $(1, 4, -3, -1) \in \ker(A)$. We note that the matrix $A$ does not admit any gluing.
\end{example}

In this section we give a complete description of the distance reducing Markov bases for monomial curves in $\mathbb A^4$ and admit a gluing. In particular, we study the glued non-complete intersections and show the circuits characterise the distance reducing property. First we give an explicit description of the complete intersection cases.

\subsection{Complete intersections}

For monomial curves in $\mathbb A^4$, there are two types of complete intersections: 
\[
(((a_1 \circ a_2) \circ a_3) \circ a_4)
\quad \text{and} \quad
((a_1 \circ a_2) \circ (a_3 \circ a_4)).
\]
In each case, the minimal Markov bases can be characterised by Theorem~\ref{thm: characterisation dist red first kind} and Theorem~\ref{thm: dist red ci implies first kind}.

\begin{corollary}\label{cor: dim 4 markov basis type 211}
    Suppose that $A = \begin{pmatrix}
        a_1 & a_2 & a_3 & a_4
    \end{pmatrix}$ is a complete intersection with type $(((a_1 \circ a_2) \circ a_3) \circ a_4)$. Suppose that $M$ is a minimal Markov basis of $A$ consisting of
    \[
    b = (b_1, -b_2, 0, 0), \quad
    c = (c_1, c_2, -c_3, 0), \quad
    d = (d_1, d_2, d_3, -d_4)
    \]
    where $b_1, b_2, c_2, d_4 > 0$ and $c_1, c_2, d_1, d_2, d_3 \ge 0$. Without loss of generality, we assume $b_1 > b_2$. Then $M$ is distance reducing if and only if each of the following conditions is satisfied:
    \begin{enumerate}
        \item[(a)] $c_1 < c_2 + c_3$,
        \item[(b)] ($c_1 = 0$ and $c_3 < c_2$) or $d_1 + d_3 < d_2 + d_4$,
        \item[(c)] $c_1 + c_2 < c_3$ or $d_1 + d_2 < d_3 + d_4$.
    \end{enumerate}
\end{corollary}

\begin{proof}
    Since $A$ admits a gluing of the first kind, by Theorem~\ref{thm: characterisation dist red first kind}, we have that $M$ is distance reducing if and only if $M$ satisfies conditions $R_{2,3}, R_{2,4}, R_{3,4}$. Let us consider condition $R_{2,3}$. Since $b_1 > b_2$, it follows that $(2,3)$-(i) does not hold. By definition the condition $(2,3)$-(ii) does not hold. So $M$ satisfies $R_{2,3}$ if and only if $M$ satisfies condition $(2,3)$-(iii), which is equivalent to $c_1 < c_2 + c_3$. Similarly, it follows that $M$ satisfies $R_{2,4}$ if and only if condition (b) above holds, and $M$ satisfies $R_{3,4}$ if and only if condition (c) above holds.
\end{proof}

\begin{example}
    Consider the matrix $A = \begin{pmatrix}
        7 & 8 & 22 & 23
    \end{pmatrix}$. This matrix admits a gluing of type $(((7 \circ_{56} 8) \circ_{22} 22) \circ_{23} 23)$ and has a unique minimal Markov basis
    \[
    b = (8, -7,  0,  0), \quad
    c = (2,  1, -1,  0), \quad
    d = (1,  2,  0, -1).
    \]
    This Markov basis does not satisfy condition (a) in the above corollary, hence it fails to reduce the distance of the circuit $(0, 11, -4, 0)$.
\end{example}

\begin{corollary}\label{prop: dim 4 markov basis type 22}
    Suppose that $A = \begin{pmatrix}
        a_1 & a_2 & a_3 & a_4
    \end{pmatrix}$ is a complete intersection with type $((a_1 \circ a_2) \circ (a_3 \circ a_4))$. Suppose that $M$ is a minimal Markov basis of $A$ consisting of
    \[
    b = (b_1, -b_2, 0, 0), \quad
    c = (0, 0, c_3, -c_4), \quad
    d = (d_1, d_2, -d_3, -d_4)
    \]
    where $b_1, b_2, c_3, c_4 > 0$ and $d_1, d_2, d_3, d_4 \ge 0$. Without loss of generality, we assume that $b_1 > b_2$ and $c_3 > c_4$. Then $M$ is distance reducing if and only if the following conditions hold:
    \begin{itemize}
        \item[(i)] At least one of $d_1$ and $d_3$ is zero,

        \item[(ii)] $d_1 + d_3 < d_2 + d_4$.
    \end{itemize}
    If $d_1 = 0$, then $A$ admits a gluing of the first kind given by $(((a_1 \circ a_2)\circ a_4)\circ a_3)$. Otherwise if $d_3 = 0$ then $A$ admits a gluing of the first kind given by $(((a_3 \circ a_4)\circ a_2)\circ a_1)$.
\end{corollary}

\begin{example}
    Consider the matrix $A = \begin{pmatrix}
        90 & 126 & 350 & 525
    \end{pmatrix}$. This matrix admits a gluing of type $(((90 \circ_{630} 126) \circ_{3150} (350 \circ_{1050} 525))$ and a minimal Markov basis
    \[
    b = ( 7, -5,  0,  0), \quad
    c = ( 0,  0,  3, -2), \quad
    d = (14, 15, -3, -4).
    \]
    Since $d_1 \neq 0$ and $d_3 \neq 0$, it follows that $M$ does not reduce the distance of the circuit $(0, 25, 0, -6)$. 
\end{example}

\subsection{Glued non-complete intersections}

Throughout this section, we consider a matrix $A$ that admits a gluing but is not a non complete intersection. The main result of this section is Theorem~\ref{thm: dim4 nci glue}, which shows that a minimal Markov basis for $A$ is distance reducing if and only if it reduces the distance of the circuits of $A$.

The gluing is given by $A = A' \circ a_4$ where $A' = \begin{pmatrix}
    a_1 & a_2 & a_3
\end{pmatrix}$
is not a complete intersection. By Theorem~\ref{thm: Herzog dim 3}, we have that $A'$ has a unique minimal Markov basis, which gives us $3$ indispensable elements for any Markov basis of $A$:
\[
g_1 = (-c_1, v_{12}, v_{13}, 0),\ 
g_2 = (v_{21}, -c_2, v_{23}, 0),\ 
g_3 = (v_{31}, v_{32}, -c_3, 0).
\]
By \cite[Theorem~1.4]{rosales1997presentations} or, more generally, \cite[Theorem~2.6]{morales2005complete}, a gluing introduces exactly one new generator. In this case, we have that any Markov basis of $A$ contains exactly one more element $h = (h_1, h_2, h_3, -h_4)$. 
We denote by $M$ the Markov basis $\{g_1, g_2, g_3, h\}$. We denote by $M'$ the unique minimal Markov basis $\{g_1', g_2', g_3'\}$ of $A'$. This gives us the following.

\begin{proposition}\label{prop: dim4 nci gluing gens}
    Let $A = \begin{pmatrix}
        a_1 & a_2 & a_3 & a_4
    \end{pmatrix}$ be a matrix and $A' = \begin{pmatrix}
        a_1 & a_2 & a_3
    \end{pmatrix}$ a submatrix of $A$. Suppose that $A'$ is not a complete intersection and $A'$ has unique minimal Markov basis 
    \[
    M' = \{
    g_1' = (-c_1, v_{12}, v_{13}),\  
    g_2' = (v_{21}, -c_2, v_{23}),\  
    g_3' = (v_{31}, v_{32}, -c_3)\}
    \]
    as in Theorem~\ref{thm: Herzog dim 3}. If $A$ admits a gluing, i.e., there exists $x \in \NN A' \cap a_4\NN$
    such that $x\ZZ = \ZZ A' \cap a_4\ZZ$, then any minimal Markov basis $M$ of $A$ is given by
    \[
    M = \{
    g_1 = (g_1', 0), \ 
    g_2 = (g_2', 0), \ 
    g_3 = (g_3', 0), \ 
    h = (h_1, h_2, h_3, -h_4)\} \subseteq \ker(A), 
    \]
    for some $h_1, h_2, h_3 \ge 0$ and $h_4 > 0$.
\end{proposition}

With the notation as above, we have the following.

\begin{proposition} \label{prop: dim4 nci reduces circuits implies reduces dim3 circuits}
    Suppose that $M$ reduces the distance of the circuits of $A$. Then $M'$ reduces the distance of the circuits of $A'$.
\end{proposition}

\begin{proof}
    Assume by contradiction that $M'$ does not reduce all circuits of $A'$. Note that the result is symmetrical with respect to the permutation of the first three columns of $A$. So without loss of generality, we may assume that the circuit $z' = (z_1, -z_2, 0)$ is not distance reduced by $M'$. Since $A$ reduces the distance of the circuit $z = (z_1, -z_2, 0, 0)$, and none of $g_1$, $g_2$, $g_3$ reduces the distance of $z$, it follows that $h$ reduces the distance of $z$. In particular, $h$ is applicable to $z$ so we must have that either: $h_1 \le z_1$ and $h_2 = h_3 = 0$; or $h_2 \le z_2$ and $h_1 = h_3 = 0$. Without loss of generality, let us assume that $h_1 \le z_1$ and $h_2 = h_3 = 0$. Since $h$ reduces the distance of $z$, we must have that $|z| > |z-h|$. So
    \[
    z_1 + z_2 = |z| > |z-h| = |(z_1 - h_1, -z_2, 0, h_4)| = z_1 - h_1 + z_2 + h_4.
    \]
    Hence $h_1 > h_4$.

    It is easy to see that $h$ does not reduce the distance of the circuit $y = (0, y_2, 0, -y_4)$. So, by assumption, the distance of $y$ is reduced by at least one of $g_1$, $g_2$, or $g_3$. However, the moves $g_1$ and $g_3$ are not applicable to $y$. So the distance of $y$ must be reduced by $g_2$. So we have $y_2 \ge c_2$ and $|y| > |y + g_2|$, hence
    \[
    y_2 + y_4 = |y| > |y+g_2| = |(v_{21}, y_2 - c_2, v_{23}, -y_4)| = v_{21} + y_2 - c_2 + v_{23} + y_4.
    \]
    It follows that $c_2 > v_{21} + v_{23}$.

    Let us now consider the move $g_2' = (v_{21}, -c_2, v_{23}) \in M'$. By Theorem~\ref{thm: Herzog dim 3}, we have that $g_2'$ is a minimal element of type $2$. Therefore $g_2'$ is applicable to the circuit $z' = (z_1, -z_2, 0)$, so $c_2 \le z_2$. Since $c_2 > v_{21} + v_{23}$, it follows that
    \[
    |z'-g_2'| = |(z_1 - v_{21}, -z_2 + c_2, -v_{23})| \le z_1 + v_{21} + z_2 - c_2 + v_{23} < |z|.
    \]
    So $g_2'$ reduces the distance of the circuit $z'$, which contradicts our original assumption. This completes the proof.
\end{proof}

\begin{theorem}\label{thm: dim4 nci glue}
    Assume $a_1 < a_2 < a_3$. A minimal Markov basis $M$ is distance reducing if and only if all of the following conditions hold:
    \begin{enumerate}
        \item[(i)] $v_{21} < c_2 + v_{23}$ or $v_{31} < v_{32} + c_3$,
        \item[(ii)] $v_{21} + v_{31} < c_2$ or $h_1 + h_3 < h_2 + h_4$,
        \item[(iii)] $h_1 + h_2 < h_3 + h_4$.
    \end{enumerate}
    Moreover,
    \begin{itemize}
        \item If $M$ reduces the distance of the circuit $(0, z_2, -z_3, 0)$ then condition (i) holds,
        \item If $M$ reduces the distance of the circuit $(0, z_2, 0, -z_4)$ then condition (ii) holds,
        \item If $M$ reduces the distance of the circuit $(0, 0, z_3, -z_4)$ then condition (iii) holds.
    \end{itemize}
    In particular, $M$ is distance reducing if and only if $M$ reduces the distance of the three circuits listed above.
\end{theorem}

\begin{proof}
    On the one hand, if $M$ is distance reducing then by Proposition~\ref{prop: dim4 nci glue forward}, we have that condition (i), (ii), and (iii) hold. Conversely, if conditions (i), (ii), and (iii) hold then by Proposition~\ref{prop: dim4 nci glue reverse}, we have that $M$ is distance reducing.
\end{proof}

We immediately obtain the following corollary from Theorem~\ref{thm: dim4 nci glue}.

\begin{corollary}
    Let $M$ be a minimal Markov basis. Then $M$ is distance reducing if and only if $M$ is distance reducing for the circuits of $A$.
\end{corollary}

We now prove Propositions~\ref{prop: dim4 nci glue forward} and \ref{prop: dim4 nci glue reverse} that appear in the proof of Theorem~\ref{thm: dim4 nci glue}.

\begin{proposition}\label{prop: dim4 nci glue forward}
    Assume the setup of Theorem~\ref{thm: dim4 nci glue}. If $M$ is distance reducing then conditions (i), (ii), and (iii) hold.
\end{proposition}

\begin{proof}
\noindent \textbf{Condition (i).} Suppose that $M$ is distance reducing. Then $M$ reduces the distances of the circuits of $A$. So by Proposition~\ref{prop: dim4 nci reduces circuits implies reduces dim3 circuits} we have that $M'$ reduces the circuits of $A'$. By Theorem~\ref{thm: dim3 nci} it follows that condition (i) holds.

\medskip

\noindent \textbf{Condition (ii).} Consider the circuit $z = (0, z_2, 0, -z_4)$ of $A$. By assumption $M$ reduces the distance of $z$. Note that neither $g_1$ nor $g_3$ are applicable to $z$, so either $g_2$ or $h$ reduces the distance of $z$. In both cases, we will show that condition (ii) holds. If $g_2$ reduces the distance, then we have $z_2 \ge c_2$ and $|z| > |z + g_2|$, hence
\[
z_2 + z_4 = |z| > |z+g_2| = |(v_{21}, z_2 - c_2, v_{23}, -z_4)| = v_{21} + z_2 - c_2 + v_{23} + z_4.
\]
Therefore, we have $v_{21} + v_{31} < c_2$ and so condition (ii) holds. On the other hand, suppose that $h$ reduces the distance of $z$. Since the 4th coordinate of $g_1$, $g_2$, $g_3$ are zero, it follows that $h_4 = \alpha z_4$ for some positive integer $\alpha$. Consider 
\[
z - \alpha h
= (-\alpha h_1, z_2 - \alpha h_2, -\alpha h_3, 0) \in \ker(A).
\]
Since $-\alpha h_1, -\alpha h_3 \le 0$ and $\ker(A) \cap \NN^4 = \{0\}$, we have that $z_2 - \alpha h_2 \ge 0$. In particular, we have $z_2 \ge h_2$. Since $h$ is applicable to $z$ we have $z_4 \ge z_4$. Since $h$ reduces the distance of $z$ we have $|z| > |z-h|$. It follows that
\[
z_2 + z_4 = |z| > |z-h| = |(-h_1, z_2-h_2, -h_3, -z_4 + h_4)|
= h_1 + z_2 - h_2 + h_3 + z_4 - h_4,
\]
and so we have $h_1 + h_3 < h_2 + h_4$. Hence, condition (ii) holds.

\medskip

\noindent \textbf{Condition (iii).} Consider the circuit $z = (0, 0, z_3, -z_4)$. By assumption $M$ reduces the distance of $z$. However, the moves $g_1$ and $g_2$ are not applicable to $z$. So either $g_3$ or $h$ reduces the distance of $z$. 

First, we show that $g_3$ does not reduce the distance of $z$. Since $a_1 < a_2 < a_3$ and $g_3 \in \ker(A)$, we have
\[
c_3 = \frac{a_1}{a_3} v_{31} + \frac{a_2}{a_3} v_{32} < v_{31} + v_{32}.
\]
So, if $g_3$ is applicable to $z$ then $z_3 \ge c_3$, hence we have
\[
|z+g_3| = |(v_{31}, v_{32}, z_3 - c_3, -z_4)| 
= v_{31} + v_{32} + z_3 - c_3 + z_4
> |z|.
\]
Therefore $g_3$ does not reduce the distance of $z$.

So we have deduced that $h$ reduces the distance of $z$. By a similar argument to the condition (ii) case, we will show that $z_3 \ge h_3$. Since the 4th coordinate of $g_1$, $g_2$, $g_3$ are zero, it follows that $h_4 = \alpha z_4$ for some positive integer $\alpha$. Consider 
\[
z - \alpha h
= (-\alpha h_1, - \alpha h_2, z_3 -\alpha h_3, 0) \in \ker(A).
\]
Since $-\alpha h_1, -\alpha h_2 \le 0$ and $\ker(A) \cap \NN^4 = \{0\}$, we have that $z_3 - \alpha h_3 \ge 0$. In particular, we have $z_3 \ge h_3$. Since $h$ is applicable to $z$ we have $z_4 \ge z_4$. Since $h$ reduces the distance of $z$ we have $|z| > |z-h|$. It follows that
\[
z_3 + z_4 = |z| > |z-h| = |(-h_1, -h_2, z_3-h_3, -z_4 + h_4)|
= h_1 + h_2 + z_3-h_3 + z_4 - h_4,
\]
and so we have $h_1 + h_2 < h_3 + h_4$. Hence, condition (iii) holds.
\end{proof}

\begin{proposition}\label{prop: dim4 nci glue reverse}
    Assume the setup of Theorem~\ref{thm: dim4 nci glue}. If conditions (i), (ii), (iii) holds, then $M$ is distance reducing.
\end{proposition}

\begin{proof}
Suppose that conditions (i), (ii), (iii) hold and assume by contradiction that there exists $z \in \ker(A)$ whose distance cannot be reduced by $M$. If $z_4 = 0$, then we may apply Theorem~\ref{thm: dim3 nci} as condition (i) holds. So one of $g_1, g_2, g_3$ reduces the distance of $z$, which is a contradiction. So we may assume that $z_4 < 0$.

We have $z = \alpha h + \beta g_1 + \gamma g_2 + \delta g_3$ for some integers $\alpha, \beta, \gamma, \delta$. Since $z_4 < 0$, it follows that $\alpha > 0$. By Proposition~\ref{prop: 3 dim nci gi sum to zero}, we have $g_1 + g_2 + g_3 = 0$. Let $\lambda = \min\{\beta, \gamma, \delta \}$, then we have
\[
z = 
\alpha h + \beta g_1 + \gamma g_2 + \delta g_3 = 
\alpha h + (\beta - \lambda) g_1 + (\gamma - \lambda) g_2 + (\delta - \lambda) g_3.
\]
Among the coefficients for $g_1$, $g_2$, $g_3$ in the right-most expression for $z$, we have that at least one is zero and the others are non-negative. So, without loss of generality, we may assume that $\beta, \gamma, \delta \ge 0$ and at least one is zero. Note that if $\beta = \gamma = \delta = 0$, then we have $z = \alpha h$, so $h$ reduces the distance of $z$, a contradiction. So at least one of $\beta, \gamma, \delta$ is nonzero. We take cases on whether $\gamma$ is zero.

\medskip

\noindent \textbf{Case 1.} Assume $\gamma = 0$.
So we have $\beta, \delta \ge 0$. It follows that
\[
z = \alpha h + \beta g_1 + \delta g_3 = 
(\alpha h_1 - \beta c_1 + \delta v_{31},\ 
\alpha h_2 + \beta v_{12} + \delta v_{32},\ 
\alpha h_3 + \beta v_{13} - \delta c_3,\ 
-\alpha h_4).
\]
By condition (ii) either $h_1 + h_3 < h_2 + h_4$ or $c_2 > v_{21} + v_{23}$. If $h_1 + h_3 < h_2 + h_4$, then we have $z_2 = \alpha h_2 + \beta v_{12} + \delta v_{32} \ge h_2$ so
\[
|z-h| = |z_1 \pm h_1| + |z_2 - h_2| + |z_3 \pm h_3| + |-z_4 + h_4|
\le |z| + h_1 + h_3 - h_2 - h_4 < |z|.
\]
Hence $h$ reduces the distance of $z$, a contradiction and we are done. 

Suppose that $c_2 > v_{21} + v_{23}$. If $\beta > 0$ and $\delta > 0$, then we have $z_2 = \alpha h_2 + \beta v_{12} + \delta v_{32} \ge v_{12} + v_{32} = c_2$. So $g_2$ is applicable to $z$ and we have 
\[
|z+g_2| = |z_1 \pm v_{21}| + |z_2-c_2| + |z_3 \pm v_{32}| + |z_4| 
\le |z| + v_{21} - c_2 + v_{32} < |z|.
\]
Hence $g_2$ reduces the distance of $z$, a contradiction. So either $\beta = 0$ or $\delta = 0$.

\medskip

\noindent \textbf{Case 1.1.} Assume $\beta = 0$ and $\delta > 0$. In this case we have
\[
z = \alpha h + \delta g_3 = (\alpha h_1 + \delta v_{31}, \alpha h_2 + \delta v_{32}, \alpha h_3 - \delta c_3, -\alpha h_4).
\]
Since $z_1 = \alpha h_1 + \delta v_{31} \ge v_{31}$ and $z_2 = \alpha h_2 + \delta v_{32} \ge v_{31}$, we have that $g_3$ is applicable to $z$. Also, since $a_1 < a_2 < a_3$, it follows that $c_3 < v_{31} + v_{32}$. So
\begin{align*}
    |z - g_3| &= |(\alpha h_1 + (\delta - 1)v_{31},\ 
        \alpha h_2 + (\delta - 1)v_{32},\ 
        \alpha h_3 - (\delta - 1)c_2,\ 
        -\alpha h_4)|\\
        &= z_1 - v_{31} + z_2 - v_{32} + |z_3 \pm c_3| + z_4 \\
        & \le z_1 + z_2 + z_3 + z_4 - v_{31} - v_{32} + c_3 \\
        & < |z|.
\end{align*}
Hence $g_3$ reduces the distance of $z$, which is a contradiction.

\medskip

\noindent
\textbf{Case 1.2.} Assume $\delta = 0$ and $\beta > 0$. In this case we have
\[
z = \alpha h + \beta g_1 = (\alpha h_1 - \beta c_1,\ 
\alpha h_2 + \beta v_{12},\ 
\alpha h_3 + \beta v_{13},\ 
-\alpha h_4).
\]
We have that $h$ is applicable to $z$, and $z_3 = \alpha h_3 + \beta v_{13} \ge h_3$. So,
\begin{align*}
    |z - h| &= |((\alpha - 1) h_1 - \beta c_1,\ 
            (\alpha - 1) h_2 + \beta v_{12},\ 
            (\alpha - 1) h_3 + \beta v_{13},\ 
            -(\alpha - 1) h_4)|\\
        &= |z_1 \pm h_1| + z_2 - h_2 + z_3 - h_3 + z_4 - h_4 \\
        & \le z_1 + z_2 + z_3 + z_4 + h_1 + h_2 - h_3 - h_4 \\
        & < |z|,
\end{align*}
where the final inequality follows from condition (iii) $h_1 + h_2 < h_3 + h_4$. Hence $h$ reduces the distance of $z$, which is a contradiction.

\medskip 

\noindent\textbf{Case 2.} Assume $\gamma > 0$. So either $\beta = 0$ or $\delta = 0$. We will show that $\delta = 0$. If $\delta > 0$, then $\beta = 0$ and we have $z = \alpha h + \gamma g_2 + \delta g_3$. If $\delta > 0$, then we have $z_1 = \alpha h_1 + \gamma v_{21} + \delta v_{31} \ge v_{21} + v_{31} = c_1$. So $g_1$ is applicable to $z$. Since $a_1 < a_2 < a_3$, it follows that $c_1 > v_{12} + v_{13}$. So $g_1$ reduces the distance of $z$, a contradiction. So, we must have $\delta = 0$. 

Since $z_4 \ge h_4$, it follows that $h$ is applicable to $z$. Since $\delta = 0$, we have $z_3 = \alpha h_3 + \beta v_{13} + \gamma v_{23} \ge h_3$. Hence, we have
\[
|z-h| = |z_1 \pm h_1| + |z_2 \pm h_2| + |z_3 - h_3| + |-z_4 + h_4| \le 
|z| + h_1 + h_2 - h_3 - h_4 < |z|,
\]
where the final inequality follows from condition (iii) $h_1 + h_2 < h_3 + h_4$. Therefore $h$ reduces the distance of $z$, a contradiction.
\end{proof}

\section{Generalisations of \texorpdfstring{$1$}{1}-norm results}\label{sec: generalisation of 1 norm results}

In this section we extend results from \cite{aoki2005distance} to the non-homogeneous case.

\begin{remark}
    Colloquially, the result\cite[Proposition 3]{aoki2005distance} says: to check whether a set is distance reducing, it suffices to check whether it reduces the distance of the Graver basis. This proof carefully uses the assumption of homogeneity of $A$ and the definition of the $1$-norm. To see this clearly, consider the main claim from the proof that $\emptyset \neq \supp(z^+) \cap \supp(z_1^-)$. This claim follows from only one key assumption:
    \begin{itemize}
        \item $z$ is a move from $z_1^+$ that reduces the distance to $z_1^-$:
        \[
        ||(z_1^+ + z) - z_1^-|| < ||z_1^+ - z_1^-|| = ||z_1||.
        \]
    \end{itemize}
    Since the move $z$ is applicable to $z_1^+$, we have that $\supp(z_1^+) \subseteq \supp(z^-)$. If we assume by contradiction that $\supp(z^+) \cap \supp(z_1^-) = \emptyset$, then $\supp(z) \cap \supp(z_1^-) = \emptyset$ and we deduce the following:
    \[
    ||(z_1^+ + z) - z_1^-||
    = 
    ||z_1^+ + z|| + ||z_1^-||
    = 
    ||(z_1^+ - z^-) + z^+|| + ||z_1^-||
    = 
    ||z_1^+|| + ||z_1^-|| = ||z_1||.
    \]
    These equalities are immediate using the definition of the $1$-norm and the final equality uses the assumption of homogeneity, which means that $||z^+|| = ||z^-||$. However, this chain of equalities contradicts the previous strict inequality.

    In Theorem~\ref{thm: 1 norm reducing test via Graver basis} we show that the condition of homogeneity can be dropped. In particular, this generalises [Proposition~3, Aoki-Takemura] to the inhomogeneous case.
\end{remark}

\begin{proposition}\label{prop: Graver is 1 norm reducing}
    The Graver basis is strongly distance reducing, even when $A$ is inhomogeneous. 
\end{proposition}

\begin{proof} The proof is identical to the proof of the homogeneous case by Aoki-Takemura, however we include it here for completeness.
Let $x, y$ be two elements of a fiber. If $x-y$ belongs to the Graver basis then we are done. So, let us assume that $x-y$ does not belong to the Graver basis. Then $x-y$ admits a proper conformal decomposition $x - y = u + v$ for some $u, v \in \ker(A)$. If $u$ does not belong to the Graver basis then $u$ admits a proper conformal decomposition $u = u' + v'$. We then have $x - y = u' + (v' + v)$ is a proper conformal decomposition. If we continue in this way, we note that $||u|| = ||u'|| + ||v'|| > ||u'||$ so there are only finitely many steps until the first term is an element of the Graver basis. So we may assume that $u$ belongs to the Graver basis.

By conformality of the decomposition, we have $||x-y|| = ||u|| + ||v||$. Also by conformality, we have that $u$ is a move that is applicable to both $x$ and $y$. So it follows that 
\[
||(x - u) - y|| = ||x - (y + u)|| = ||v|| < ||x-y||.
\]
Hence the Graver basis is strongly distance reducing.
\end{proof}

The next example shows that the property of distance reduction is not transitive. For instance, if $M$ is a Markov basis and $B$ is distance reducing for $M$. Then it is not true that $B$ is a Markov basis.

\begin{example}
    Let $A = \begin{pmatrix} 2 & 3 & 4 \end{pmatrix}$.
    Consider the Markov basis 
    $M = \{(3,-2,0), (2,0,-1) \}$.
    The singleton $B = \{(2, 0, -1)\}$ is clearly not a Markov basis yet it reduces the distance of both elements of $M$. To see this, note that $(2,0,-1)$ is applicable to $(3,-2,0)$ and
    \[
    ||(3,-2,0) - (2,0,-1)|| = ||(1,-2,1)|| = 4 < 5 = ||(3,-2,0)||. 
    \]
\end{example}

The following theorem shows that, to check whether a set is distance reducing, it suffices to check that it distance reduces the Graver basis.

\begin{theorem}\label{thm: 1 norm reducing test via Graver basis}
    Let $B \subseteq \ker(A)$ be any subset. Then $B$ is distance reducing if and only if it is distance reducing for the Graver basis $G(A)$. Similarly, $B$ is strongly distance reducing if and only if it is strongly distance reducing for $G(A)$.
\end{theorem}

\begin{proof}
    Clearly if $B$ is distance reducing then it is distance reducing for the Graver basis. On the other hand, suppose that $B$ reduces the distance of the Graver basis. For any $z \in \ker(A) \backslash G(A)$, we have that $z$ admits a proper conformal decomposition $z = g + h$ where $g \in G(A)$. Since $B$ reduces the distance of $G(A)$, there is a move $b \in B$ that distance reduces $g$. By \Cref{lem: 1-norm reduction and conformal decomposition}, it follows that the move $b$ also distance reduces $z$. So $B$ is distance reducing and we are done.
\end{proof}

\begin{remark}
    The above theorem shows that a minimal Markov basis $M$ is distance reducing if and only if $M$ is distance reducing for the Graver basis.
\end{remark}

Recall \Cref{prop: B d-reducing => Markov basis}, which shows that if a set is distance reducing then it is Markov basis. The following corollary shows that the sets that distance reduces the Graver basis are Markov bases.

\begin{corollary}
    Suppose that a set of moves $B$ is distance reducing for the Graver basis, then $B$ is a Markov basis.   
\end{corollary}

\begin{proof}
    By Theorem~\ref{thm: 1 norm reducing test via Graver basis}, if $B$ is distance reducing for the Graver basis then it is distance reducing. So by Proposition~\ref{prop: B d-reducing => Markov basis}, it follows that $B$ is a Markov basis.
\end{proof}

\begin{lemma} \label{lem: 1-norm reduction and conformal decomposition}
    Suppose that $b \in \ker(A)$ is a move that distance reduces an element $x \in \ker(A)$. Suppose that $z \in \ker(A)$ has a conformal decomposition $z = x + y$ for some $y \in \ker(A)$. Then $b$ distance reduces $z$.
\end{lemma}

\begin{proof}
    Without loss of generality, assume that $||x - b|| < ||x||$ and $x^+ \ge b^+$. Note that the conformal decomposition gives us that $||z|| = ||x|| + ||y||$ and $z^+ \ge b^+$ so we have
    \[
    ||z - b|| = ||x-b + y|| \le ||x-b|| + ||y|| < ||x|| + ||y|| = ||z||.
    \]
    Hence $b$ reduces the distance of $z$.
\end{proof}

\section{Distance irreducible elements}\label{sec: dist irred elements}

The distance irreducible elements of $\ker(A)$ are the set of moves that cannot be distance reduced. We compare the distance irreducible elements with the indispensable set and Graver basis by using the notions of semi-conformal decomposition and conformal decomposition. 

\begin{definition}
    Let $z \in \ker(A)$ and suppose that there exist $u,v \in \ker(A) \backslash \{0\}$ such that $z = u+v$. We recall that:
    \begin{itemize}
        \item $z = u+v$ is a \textit{proper conformal decomposition} and write $z = u +_c v$ if 
        \[
        \supp(u^+) \cup \supp(v^+) = \supp(z^+)
        \quad \text{and} \quad
        \supp(u^-) \cup \supp(v^-) = \supp(z^-)
        \]

        \item $z = u + v$ is a \textit{proper semi-conformal decomposition} and write $z = u +_{sc} v$ if 
        \[
        \supp(u^+) \subseteq [n] \backslash \supp(v^-)
        \quad \text{and} \quad
        \supp(v^-) \subseteq [n] \backslash \supp(u^+)
        \]
    \end{itemize}
    We introduce the following definitions:
    \begin{itemize}
        \item $z = u + v$ is a \textit{positive distance decomposition} and write $z = u +_{pd} v$ if
        \[
        u^+ \le z^+ 
        \quad \text{and} \quad
        ||v|| < ||z||
        \]
        \item $z = u + v$ is a \textit{negative distance decomposition} and write $z = u +_{nd} v$ if
        \[
        u^- \le z^-
        \quad \text{and} \quad
        ||v|| < ||z||
        \]
    \end{itemize}
\end{definition}

We define classes of indecomposable elements of $\ker(A)$ with respect to the above decompositions.

\begin{definition}\label{def: distance irreducible elements}
    We recall the indispensable and primitive moves:
    \begin{itemize}
        \item $S(A) = \{z \in \ker(A) \colon z \text{ has no semi-conformal decomposition} \}$ \textit{indispensable set}

        \item $G(A) = \{ z \in \ker(A) \colon z \text{ has no conformal decomposition} \}$ \textit{primitive elements}
    \end{itemize}

    We define the following sets of moves:
    \begin{itemize}
        \item $D^+(A) = \{ z \in \ker(A) \colon z \text{ has no positive distance decomposition}\}$ 

        \item $D^-(A) = \{ z \in \ker(A) \colon z \text{ has no negative distance decomposition}\}$

        \item $D^w(A) = D^+(A) \cup D^-(A)$ \textit{weakly distance irreducible elements}

        \item $D(A) = D^+(A) \cap D^-(A)$ \textit{distance irreducible elements}
    \end{itemize}
\end{definition}

\begin{proposition}\label{prop: dist irred is symmetric}
    The following hold:
    \[
        D^+(A) = -D^-(A), \quad 
        D(A) = -D(A), \quad
        D^w(A) = -D^w(A).
    \]
\end{proposition}

\begin{proof}
    Suppose that $z = u +_{pd} v$ is a positive distance decomposition. Then it follows that $-z = (-u) +_{nd} (-v)$ is a negative distance decomposition since
    \[
        (-u)^- = u^+ \le z^+ = (-z)^-
        \text{ and }
        ||-v|| = ||v|| \le ||z|| = ||-z||.
    \]
    By a similar argument, we have that if $z$ admits a negative distance decomposition then $-z$ admits a positive distance decomposition. Therefore $D^+(A) = -D^-(A)$. The remaining equalities follow immediately from the definition.
\end{proof}

Distance irreducible elements for distance reducing Markov bases are the analogue of indispensable elements for Markov bases. 

\begin{proposition}
    Suppose that $B \subseteq \ker(A)$ is distance reducing, then $D(A) \subseteq B$. If $B$ is strongly distance reducing then $D(A)^w \subseteq B$.
\end{proposition}

\begin{proof}
    Fix a distance reducing set $B \subseteq \ker(A)$ and any $z \in D(A)$. Since $B$ is distance reducing, there exists an element $b \in B$ that reduces the distance of $z$. So we have that $b$ is applicable to $z$ and reduces its distance. So at least one of the following holds:
    \[
        {\rm (i)} \ b^+ \le z^+, \quad
        {\rm (ii)} \ b^- \le z^+, \quad
        {\rm (iii)} \ b^+ \le z^-, \quad
        {\rm (iv)} \ b^- \le z^-.
    \]
    
    \smallskip \noindent
    \textbf{Case (i)} Suppose that $b^+ \le z^+$. Since $b$ reduces the distance of $z$, we have $||z-b|| < ||z||$. If $z-b \neq 0$, then we have $z = b +_{pd} (z-b)$ is a positive distance decomposition, which contradicts $z \in D(A) \subseteq D^+(A)$. Hence $z = b \in B$.

    \smallskip \noindent
    \textbf{Case (ii)} Suppose that $b^- \le z^+$. Since $b$ reduces the distance of $z$ we have $||z + b|| < ||z||$. If $z \neq -b$, then we have that $z = (-b) +_{pd} (z+b)$ is a positive distance decomposition, which contradicts $z \in D(A) \subseteq D^+(A)$. So $-z = b \in B$.

    \smallskip \noindent
    \textbf{Cases (iii) and (iv)} These cases follow almost identically to the first two. If $z \neq b$ and $z \neq -b$ respectively then we deduce that $z \in D^-(A)$. So we have that $z \in B$ and $-z \in B$, respectively.

    Now suppose that $B \subseteq \ker(A)$ is a strongly distance reducing set and let $z \in D^w(A)$ be any element. So there exist two elements $b, b' \in B$ such that: the move $b$ is applicable to $z^+$; the move $b'$ is applicable to $z^-$; and both $b, b'$ reduce the distance of $z$. So, using the above description, we have that $b$ satisfies (i) or (ii), and $b'$ satisfies (iii) or (iv). Since $z \in D^w(A) = D^+(A) \cup D^-(A)$, we have that $z \in D^+(A)$ or $z \in D^-(A)$. If $z \in D^+(A)$, then by cases (i) and (ii) we have that $\pm z = b \in B$. Otherwise, if $z \in D^-(A)$, then by cases (iii) and (iv) we have that $\pm z = b \in B$. 
\end{proof}

\begin{proposition}
    Let $\mathcal B$ be the set of minimal distance  reducing Markov bases and $\mathcal B^s$ the set of minimal strongly distance reducing Markov bases. Then we have
    \[
    D(A) = \bigcap_{B \in \mathcal B} B
    \quad \text{and} \quad
    D^w(A) = \bigcap_{B \in \mathcal B^s} B.
    \]
\end{proposition}

\begin{proof}
    We begin by showing $D(A) = \bigcap_{B \in \mathcal B} B$. For each $B \in \mathcal B$ we have that $D(A) \subseteq B$ so $D(A) \subseteq \bigcap_{B \in \mathcal B} B$. For the opposite inclusion, it suffices to show that for each $m \in \bigcup_{B \in \mathcal B} B$ such that $m \notin D(A)$, there exists $B \in \mathcal B$ such that $m \notin B$. To do this, consider the set $\ker(A) \backslash \{m, -m\}$. This set is distance reducing because the only elements it does not contain are  $m$ and $-m$, both of which admit either a positive or negative distance reducing decompositions. Let $B$ be a minimal distance reducing subset of $\ker(A) \backslash \{m, -m\}$. We have that $B \in \mathcal B$ and $m \notin B$, hence we have shown the opposite inclusion and we are done.

    The proof that $D^w(A) = \bigcap_{B \in \mathcal B^s} B$ follows similarly. Note that for each $B \in \mathcal B^s$ we have that $D^w(A) \subseteq B$ so $D^w(A) \subseteq \bigcap_{B \in \mathcal B^s} B$. For the opposite inclusion, it suffices to show that for each $m \in \bigcup_{B \in \mathcal B^s} B$ such that $m \notin D^w(A)$, there exists $B \in \mathcal B^s$ such that $m \notin B$. To do this, consider the set $\ker(A) \backslash \{m, -m\}$. This set is strongly distance reducing because the only elements it does not contain are  $m$ and $-m$, both of which admit positive and negative distance reducing decompositions. Let $B$ be a minimal distance reducing subset of $\ker(A) \backslash \{m, -m\}$. We have that $B \in \mathcal B^s$ and $m \notin B$, hence we have shown the opposite inclusion. This concludes the proof.
\end{proof}

\begin{proposition}\label{prop: dist irred chain inclusions}
    The following chain of inclusions holds:
    \[
    S(A) \subseteq D(A) \subseteq D^w(A) \subseteq G(A).
    \]
\end{proposition}

\begin{proof}
    Throughout the proof, for a set $B \subseteq \ker(A)$, we write $\overline B := \ker(A) \backslash B$ for its complement in $\ker(A)$. We first show $D^w(A) \subseteq G(A)$ by proving that $\overline{G(A)} \subseteq \overline{D^w(A)}$. Suppose that $z \in \ker(A)$ does not belong to $G(A)$. Then we have that $z = u +_c v$ admits a proper conformal decomposition, hence $v^+ \le z^+$ and $v^- \le z^-$. In addition, we have that $u \neq 0$ and so it follows that $||v|| < ||z||$. By the conformal decomposition, we have $u^+ \le z^+$ and $u^- \le z^-$, which gives us $z = u +_{pd} v$ and $z = u +_{nd} v$ respectively. Therefore $z \in \overline{D^+(A)} \cap \overline{D^-(A)} = \overline{D^w(A)}$ and we have shown $D^w(A) \subseteq G(A)$.

    \smallskip
    Note that the inclusion $D(A) \subseteq D^w(A)$ follows immediately from the definition. So it remains to show that $S(A) \subseteq D(A)$. To do this, we show that $\overline{D(A)} \subseteq \overline{S(A)}$. Suppose that $z \in \overline{D^+(A)}$, then $z$ admits a positive distance decomposition $z = u +_{pd} v$. In particular, we have $u^+ \le z^+$. Assume, by contradiction, that $z = u + v$ is not a proper semi-conformal decomposition. Then there exists an index $i \in [n]$ such that $u_i > 0$ and $v_i < 0$. But, this gives us $z_i = u_i + v_i < u_i$, hence $u^+ \nleq z^+$, a contradiction. So $z = u +_{sc} v$ is a semi-conformal decomposition, hence $\overline{D^+(A)} \subseteq \overline{S(A)}$. Similarly, if $z \in \overline{D^-(A)}$ then there exists a negative distance decomposition $z = u +_{nd} v$. In particular, we have $u^- \le z^-$. Assume by contradiction that $z = v+u$ is not a semi-conformal decomposition. Then, as above, there exists an index $i \in [n]$ such that $u_i < 0$ and $v_i > 0$. But, we have $-z_i = -u_i -v_i < -u_i$, hence $u^- \nleq z^-$, a contradiction. So $z = v +_{sc} u$ is a semi-conformal decomposition, hence $\overline{D^-(A)} \subseteq \overline{S(A)}$. So we have that $\overline{D(A)} = \overline{D^+(A)} \cup \overline{D^-(A)} \subseteq \overline{S(A)}$, which concludes the proof.
\end{proof}

Note that for each minimal Markov basis $M$ we have the following chain of inclusions:
\[
    S(A) \subseteq M \subseteq M(A) \subseteq G(A)
\]
where $M(A)$ is the universal Markov basis. The following examples show that there is little connection between $D^w(A)$ and $M(A)$ or between $D(A)$ and $M(A)$.

\begin{example}\label{example: A = 2_3_4 continued}
    Consider the matrix $A = \begin{pmatrix}
        2 & 3 & 4 
    \end{pmatrix}$. In this case we have
    \[
    D^+(A) = \begin{bmatrix}
        2&0&-1\\
        1&-2&1\\
        -2&0&1
    \end{bmatrix}
    \text{ and }
    D^-(A) = 
    \begin{bmatrix}
        2&0&-1\\
        -2&0&1\\
        -1&2&-1
    \end{bmatrix}.
    \]
    So we have
    \[
    D(A) = \pm 
    \begin{bmatrix}
        2&0&-1
    \end{bmatrix}
    \text{ and }
    D^w(A) = D(A) \cup \pm 
    \begin{bmatrix}
        1&-2&1
    \end{bmatrix}.
    \]
    On the other hand there are two minimal Markov bases
    \[
    M_1 = \begin{bmatrix}
        2&0&-1\\
        1&-2&1
    \end{bmatrix}
    \text{ and }
    M_2 = \begin{bmatrix}
        2&0&-1\\
        3&-2&0
    \end{bmatrix}.
    \]
    So in this example we have $D^w(A) \subsetneq M(A)$
\end{example}

\begin{example}
    Consider the matrix 
    $A = \begin{pmatrix}
        8 & 14 & 15 & 20
    \end{pmatrix}$.
    We have
    \[
    D^+(A) = 
    \begin{bsmallmatrix}
    5&0&0&-2\\
    2&1&-2&0\\
    0&0&4&-3\\
    1&-2&0&1\\
    -5&0&0&2\\
    -2&-1&2&0\\
    -1&2&0&-1
    \end{bsmallmatrix}
    \text{ and }
    D^-(A) = 
    \begin{bsmallmatrix}
    5&0&0&-2\\
    2&1&-2&0\\
    1&-2&0&1\\
    -5&0&0&2\\
    -2&-1&2&0\\
    0&0&-4&3\\
    -1&2&0&-1
    \end{bsmallmatrix}.
    \]
    Hence
    \[
    D(A) = \pm
    \begin{bmatrix}
        5&0&0&-2\\
        2&1&-2&0\\
        1&-2&0&1
    \end{bmatrix}
    \text{ and }
    D^w(A) = D(A) \cup \pm 
    \begin{bmatrix}
        0&0&4&-3
    \end{bmatrix}.
    \]
    Note that $A$ is a complete intersection with gluing type $(((8 \circ 20) \circ 14) \circ 15)$. We have that $A$ has a unique minimal Markov basis so
    \[
    S(A) = M(A) = \pm \begin{bmatrix}
        5 & 0 & 0 & -2 \\
        1 & -2 & 0 & 1 \\
        2 & 1 & -2 & 0
    \end{bmatrix}
    = D(A) \subsetneq D^w(A)
    \]
    Let us consider the submatrix $A' = \begin{pmatrix}
        8&14&20
    \end{pmatrix}$.
    We have \[
    D^-(A') = D^+(A') = \pm \begin{bmatrix}
        1&-2&1\\
        5&0&-2\\
    \end{bmatrix}.
    \]
    Moreover, we have that $A'$ has a unique minimal Markov basis that coincides with the above. So $S(A) = M(A) = D(A) = D^w(A)$.
    
\end{example}

\begin{example}\label{example: A = 8_31_33_53}
    Let $A = \begin{pmatrix}
        8 & 31 & 33 & 53 
    \end{pmatrix}$. Its minimal Markov bases are
    \[
    M_1 = \begin{bsmallmatrix}
        2&-2&3&-1\\
        3&2&-1&-1\\
        5&-3&0&1\\
        5&0&2&-2\\
        6&1&-4&1\\
        8&-1&-1&0
    \end{bsmallmatrix}
    \text{ and }
    M_2 = \begin{bsmallmatrix}
        1&4&-4&0\\
        2&-2&3&-1\\
        3&2&-1&-1\\
        5&-3&0&1\\
        5&0&2&-2\\
        8&-1&-1&0
    \end{bsmallmatrix}.
    \]
    The distance irreducible elements are given by
    \[
    D(A) = D^w(A) = \pm \begin{bsmallmatrix}
        8&-1&-1&0\\
        5&0&2&-2\\
        3&2&-1&-1\\
        2&-2&3&-1\\
        5&-3&0&1\\
        1&4&-4&0\\
        3&-1&-3&2\\
        0&3&2&-3
    \end{bsmallmatrix}.
    \]
    In particular we see that $(3, -1, -3, 2) \in D(A) \backslash M(A)$.
\end{example}

To show that $D^w(A) \subseteq G(A)$, we use the characterisation of the Graver basis $G(A)$ in terms of conformal decompositions. It may be tempting to use the characterisation of the universal Markov basis $M(A)$ in terms of \textit{strongly semi-conformal decompositions} to investigate its connections with $D(A)$ and $D^w(A)$. However, the following example shows that such comparisons are not always possible.

\begin{example}\label{example: A=3_5_8_11}
    Consider the complete intersection $A =  \begin{pmatrix} 3 & 5 & 8 & 11 \end{pmatrix}$, which has two distinct minimal Markov bases:
    \[
        M_1 = \begin{bmatrix}
            1&1&-1&0\\
            2&1&0&-1\\
            5&-3&0&0
            \end{bmatrix}
        \text{ and }
        M_2 = \begin{bmatrix}
            1&1&-1&0\\
            1&0&1&-1\\
            5&-3&0&0
            \end{bmatrix}.
    \]
    In this case we have
    \[
    D^+(A) = \begin{bsmallmatrix}
        1&1&-1&0\\
        5&-3&0&0\\
        3&-4&0&1\\
        1&0&1&-1\\
        1&-5&0&2\\
        -1&-1&1&0\\
        -5&3&0&0\\
        -1&0&-1&1
    \end{bsmallmatrix} \text{ and }
    D^-(A) = \begin{bsmallmatrix}
        1&1&-1&0\\
        5&-3&0&0\\
        1&0&1&-1\\
        -1&-1&1&0\\
        -5&3&0&0\\
        -3&4&0&-1\\
        -1&0&-1&1\\
        -1&5&0&-2
    \end{bsmallmatrix}.
    \]
    And so 
    \[
    D(A) = \pm
    \begin{bmatrix}
        1 & 1 & -1 & 0 \\
        5 & -3 & 0 & 0 \\
        1 & 0 & 1 & -1 
    \end{bmatrix} \text{ and }
    D^w(A) = D(A) \cup \pm \begin{bmatrix}
        3 & -4 & 0 & 1 \\
        1 & -5 & 0 & 2
    \end{bmatrix}.
    \]
    We observe that the sets $D^w(A)$ and $M(A)$ are incomparable with respect to inclusion. In particular we have $(3, -4, 0, 2)$ and $(1, -5, 0, 2)$ lie in $D^w(A) \backslash M(A)$.

    A \textit{strongly semi-conformal decomposition} for $(3, -4, 0, 2)$ is given by
    \[
    (3, -4, 0, 1) 
    = (-1, 0, -1, 1) + (4, -4, 1, 0) 
    = (-1, 0, -1, 1) + (-1, -1, 1, 0) + (5, -3, 0, 0).
    \]
    Note that neither of these strongly semi-conformal decompositions give rise to a positive distance decomposition.
\end{example}

We now consider the properties of elements of $\ker(A)$ that cannot be distance reduced by elements of $D(A)$.

\begin{proposition}
    Let $I(A) \subseteq G(A)$ be the set of moves that cannot be distance reduced by $D(A)$. If $z \in \ker(A) \backslash G(A)$ cannot be distance reduced by any element of $D(A)$, then for any conformal decomposition $z = g_1 + \dots + g_k$ where $g_i \in G(A)$ for each $i \in [k]$, then we have that $g_i \in I(A)$ for each $i \in [k]$.
\end{proposition}

\begin{proof}
    Follows directly from \Cref{lem: 1-norm reduction and conformal decomposition}, i.e., if there exists a conformal decomposition $z = g + h$ where $g \in I(A)$ then, by the lemma, we have that $z$ is distance reduced by $D(A)$. 
\end{proof}

Analogously to the universal Markov basis $M(A)$, we define the universal distance-reducing Markov basis.

\begin{definition}
    The \textit{universal distance-reducing Markov basis} $\mathcal D(A)$ is the union of all minimal distance reducing Markov bases. The \textit{universal strongly distance-reducing Markov basis} $\mathcal D^s(A)$ is the union of all minimal strongly distance reducing Markov bases.
\end{definition}

From the definition, we have $D(A) \subseteq \mathcal D(A)$ and $D^w(A) \subseteq \mathcal D^s(A)$.

\begin{proposition}\label{prop: unique min Markov dist irred sets}
    Suppose that $A$ has a unique minimal Markov basis $M$. If $M$ is distance reducing, then we have
    \[
    S(A) = D(A) = M(A) = \mathcal D(A).
    \]
\end{proposition}

\begin{proof}
    Since $A$ has a unique minimal Markov basis it follows that $S(A) = M(A) \subseteq D(A) \subseteq \mathcal D(A)$. Suppose that $B$ is any minimally distance reducing Markov basis. Then it follows that $S(A) \subseteq D(A) \subseteq B$. By assumption $S(A) = M(A)$ is a minimal Markov basis that is distance reducing, hence any element of $B \backslash S(A)$ is redundant for the purpose of distance reduction. Hence $B = S(A)$. We have shown that there is a unique minimally distance reducing Markov basis, hence $\mathcal D(A) = S(A)$ and we are done.
\end{proof}

In the next example, we show how to compute the universal distance-reducing Markov basis.

\begin{example}
    Consider the matrix $A = \begin{pmatrix}
        3 & 5 & 11
    \end{pmatrix}$ from Example~\ref{example: 3-dim CI: unique MB does not imply reducing}. We compute the universal distance reducing Markov basis as follows. First, we have that
    \[
    D(A) = M(A) = \pm \begin{bmatrix}
        2 & 1 & -1 \\
        5 & -3 & 0 
    \end{bmatrix}
    \text{ and } 
    D^w(A) = D(A) \cup \pm \begin{bmatrix}
        3 & -4 & 1 \\
        1 & -5 & 2
    \end{bmatrix}.
    \]
    So, every distance reducing Markov basis contains $D(A)$. By Theorem~\ref{thm: 1 norm reducing test via Graver basis} we have that $D(A)$ is not distance reducing because it fails to distance reduce $(1, -5, 2)$ and $(0, 11, -5)$.

    The only elements of $\ker(A)$ that distance reduce $(1, -5, 2)$ are $(1, -5, 2)$ and $(3, -4, 1)$. Each of these elements distance reduce $(0, 11, -5)$. So, by Theorem~\ref{thm: 1 norm reducing test via Graver basis}, the minimally distance reducing Markov bases are 
    \[
    \begin{bmatrix}
        2 & 1 & -1 \\
        5 & -3 & 0 \\
        1 & -5 & 2
    \end{bmatrix}
    \text{ and }
    \begin{bmatrix}
        2 & 1 & -1 \\
        5 & -3 & 0 \\
        3 & -4 & 1
    \end{bmatrix},
    \text{ hence }
    \mathcal D(A) = D^w(A) = \begin{bmatrix}
        2 & 1 & -1 \\
        5 & -3 & 0 \\
        1 & -5 & 2 \\
        3 & -4 & 1
    \end{bmatrix} \subseteq G(A).
    \]
\end{example}

The following example shows that the universal distance reducing Markov basis need not lie inside the Graver basis. 

\begin{example}\label{example: A = 3_5_8_11 continued 2}
    Consider the matrix $A = \begin{pmatrix}
        3 & 5 & 8 & 11
    \end{pmatrix}$
    from Example~\ref{example: A=3_5_8_11}. The elements of the Graver basis that cannot by distance reduced by $D(A)$ are given by the set $I = \{(1, -5, 0, 2), (0, 11, 0, -5)\}$. To extend $D(A)$ to a minimally distance reducing Markov basis, we must add either one or two elements to distance reduce $I$.

    The only moves in $\ker(A)$ that simultaneously distance reduce both elements of $I$ are
    \[
    T_0 = \pm \begin{bmatrix}
        3 & -4 & 0 & 1 \\
        2 & -5 & 1 & 1 \\
        1 & -5 & 0 & 2
    \end{bmatrix} \subseteq G(A).
    \]
    So, there are three minimally distance reducing Markov bases given by $D(A)$ together with one element from $T_0$. 
    
    The remaining minimally distance reducing Markov bases are those that consist of $D(A)$ together with a pair of elements $z_1$ and $z_2$ such that: $z_1$ reduces the distance of $(1, -5, 0, 2)$ and does not reduce the distance of $(0, 11, 0, -5)$; and $z_2$ reduces the distance of $(0, 11, 0, -5)$ and does not reduce the distance of $(1, -5, 0, 2)$. Let $T_1$ and $T_2$ denote the set of such $z_1$ and $z_2$ respectively. These sets are given by
    \[
    T_1 = \begin{bmatrix}
        4 & -4 & 1 & 0 \\
        2 & -5 & 2 & 0 
    \end{bmatrix} 
    \text{ and }
    T_2 = \begin{bsmallmatrix}
       0&6&-1&-2\\
       1&-6&2&1\\
       2&-6&3&0\\
       1&6&0&-3\\
       0&5&1&-3\\
       0&7&-3&-1\\
       1&-7&4&0\\
       0&8&-5&0\\
       \end{bsmallmatrix}
       \cup
       \begin{bsmallmatrix}
       0&11&0&-5\\
       0&4&3&-4\\
       0&3&5&-5\\
       1&5&2&-4\\
       2&6&1&-4\\
       1&4&4&-5\\
       3&7&0&-4\\
       2&5&3&-5\\
       \end{bsmallmatrix}
       \cup
       \begin{bsmallmatrix}
       2&-10&0&4\\
       4&-9&0&3\\
       3&6&2&-5\\
       6&-8&0&2\\
       1&-11&1&4\\
       3&-10&1&3\\
       5&-9&1&2\\
       4&7&1&-5\\
       7&-8&1&1\\
    \end{bsmallmatrix}
    \cup
    \begin{bsmallmatrix}
       2&-11&2&3\\
       4&-10&2&2\\
       6&-9&2&1\\
       5&8&0&-5\\
       3&-11&3&2\\
       5&-10&3&1\\
       4&-11&4&1\\
       5&-11&5&0\\
       11&-11&0&2
    \end{bsmallmatrix}.
    \]
    So, there are sixty eight minimally distance reducing Markov bases of this form. In $T_2$, the elements: $(1, 5, 2, -4)$, $(2, 6, 1, -4)$, \dots lie outside of the Graver basis. So $\mathcal D(A) \not\subseteq G(A)$.
\end{example}

\begin{remark}\label{rmk: dist irred poset}
    The subsets $G(A)$, $M(A)$, $S(A)$, $D(A)$, $D^w(A)$, $\mathcal D(A)$, $\mathcal D^s(A)$ and Markov basis $M$ for $A$, are naturally ordered by inclusion into a poset any matrix $A$. The Hasse diagram for this poset is shown in Figure~\ref{fig: irreducible elements inclusion poset}. The dotted lines indicate pairs of sets that are incomparable for certain matrices $A$.
    
    Note that some of the inclusions are not clear. For instance, to prove that $\mathcal D(A) \subseteq \mathcal D^s(A)$, it suffices to show that any minimal distance reducing set can be extended to a minimal strongly distance reducing set. But this is not immediate. If $B$ is minimal distance reducing, then it may be impossible to extend it to a minimal strongly distance reducing Markov basis. There may exist an element $b \in B$ such that for every strongly distance reducing set $\overline B \supseteq B$ containing $B$, we may have $\overline B \backslash \{b\}$ is strongly distance reducing. In such a case, we cannot conclude that $b \in \mathcal D^s(A)$.
    
    \begin{figure}[t]
        \centering
        \includegraphics[scale = 0.8]{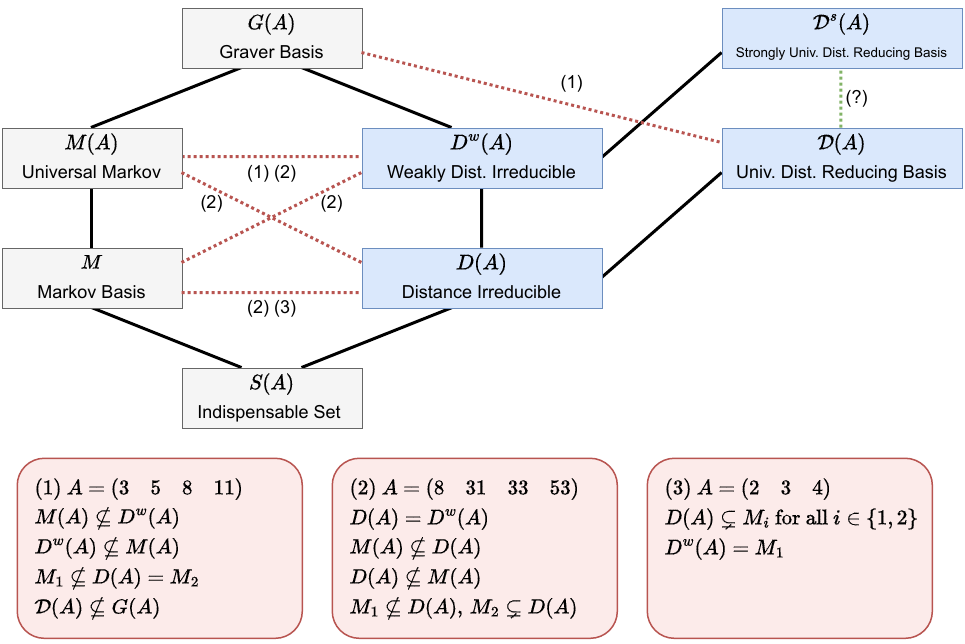}
        \caption{The Hasse diagram for a poset of subsets of $\ker(A)$. The dotted lines known incomparable elements in the poset. In particular, the examples are explained in: 
        (1) Examples~\ref{example: A=3_5_8_11} and \ref{example: A = 3_5_8_11 continued 2}, 
        (2) Example~\ref{example: A = 8_31_33_53}, 
        and (3) Example~\ref{example: A = 2_3_4 continued}.}
        \label{fig: irreducible elements inclusion poset}
    \end{figure}
\end{remark}

\begin{corollary}\label{cor: universal dist red bases are finite}
    The sets $\mathcal D(A)$ and $\mathcal D^s(A)$ are finite.
\end{corollary}

\begin{proof}
    We show that there exists $N \in \NN$ such that $\mathcal D(A), \mathcal D^s(A) \subseteq \{z \in \ker(A) \colon ||z|| \le N\}$. Let $G$ be the Graver basis of $A$ and define $\max(G) := \max\{||g|| \colon g \in G \}$. Let $N = 2\max(G)$. For any $z \in \ker(A)$ with $||z|| > N$ and $g \in G$, we have 
    \[
    \min \{||g+z||, ||g-z||\} \ge \big\lvert ||g|| - ||z||\big\rvert > |2\max(G) - \max(G)| = \max(G) \ge ||g||.
    \]
    Therefore $z$ does not distance reduce any element of the Graver basis. So, by Theorem~\ref{thm: 1 norm reducing test via Graver basis}, we have that $z$ does not belong to $\mathcal D(A)$ or $\mathcal D^s(A)$. Hence $\mathcal D(A)$ and $\mathcal D^s(A)$ are finite.
\end{proof}

\section{Discussion} \label{sec: discussion}
For monomial curves in $\mathbb A^3$ and $\mathbb A^4$, and complete intersections we have seen that the distance reducing property is characterised by the circuits. In particular, the characterisation of the distance reducing property by circuits is preserved by gluing, see Sections~\ref{sec: prelim ci and gluing}, \ref{sec: dim3 monomial curves}, \ref{sec: ci mon curves}, and \ref{sec: dim4 mon curves}.

\medskip
\noindent \textbf{Notation.}
We say that a matrix is \textit{irreducible} if it does not admit a gluing. Every non-irreducible matrix $A$ admits a gluing into two \textit{components} $A = A_1 \circ A_2$. Given a sequence of gluings $((A_1 \circ A_2) \circ A_3) \circ \dots $, the matrices $A_1, \dots, A_n$ are called its \textit{components}. A sequence of gluings is \textit{maximal} if each component is irreducible. The \textit{size} of a component is the number of columns of the component.

\smallskip

We have seen that the circuits characterise the distance reduction property: if all components have size one, see Corollary~\ref{cor: ci dist red iff dist red circuits}; and for monomial curves in $\mathbb A^4$ that are glued non complete intersections, see Theorem~\ref{thm: dim4 nci glue}. Below, we formulate generalisations about the distance reduction property.

\begin{conjecture}
    Suppose that $A$ is a $1 \times n$ matrix that admits a gluing with minimal components have size at most $3$. Then a minimal Markov basis $M$ for $A$ is distance reducing if and only if $M$ is distance reducing for the circuits of $A$.
\end{conjecture}

\begin{definition}
    A matrix $A$ has the \textit{circuit reduction property} if, for any minimal Markov basis $M$ of $A$, we have:
    \[
    M \text{ reduces the distance of the circuits of } A \implies
    M \text{ is distance reducing}.
    \]
    The set of all matrices with the circuit reduction property is denoted $\mathcal A$.
\end{definition}

\begin{conjecture}
    The set $\mathcal A$ is closed under gluing, i.e., if $A, B \in \mathcal A$
    admit a gluing then 
    $A \circ B \in \mathcal A$.
\end{conjecture}

\begin{question}
    Is it possible to characterise the distance reduction property for non complete intersection monomial curves that do not admit a gluing?
\end{question}

\begin{question}
    Is there a connection between $M(A)$ and $\mathcal D(A)$ or $M(A)$ and $\mathcal D^s(A)$? Is there an algebraic or combinatorial description for the moves in $\mathcal D(A)$, which is analogous to strongly semi-conformal for universal Markov basis elements?
\end{question}

\begin{question}
    In the definition of positive (resp. negative) distance decompositions $z = u +_{pd} v$, if we assume that $u \in G(A)$ is an element of the Graver basis, then do we obtain the same set of distance irreducible elements?
\end{question}

Just as the distance reducing minimal Markov bases are characterised by certain inequalities, see Theorems~\ref{thm: dim3 ci} and \ref{thm: characterisation dist red first kind}, we expect that a characterisation of when the distance irreducible elements $D(A)$ form a Markov basis is also governed by inequalities.  

\begin{conjecture}
    Let $A = \begin{pmatrix}
    a_1 & a_2 & a_3
\end{pmatrix}$ be a complete intersection. Suppose that $A$ has a unique minimal Markov basis 
    \[
    M = \{b := (b_1, -b_2, 0), \, c := (c_1, c_2, -c_3)\} 
    \]
    with $b_1 > b_2$. If $M$ is not distance reducing, then $D(A) = M$ if and only if $2 c_1 < b_1 + b_2$.
\end{conjecture}

\begin{example}\label{example: A = 4_9_37}
    Let $A = \begin{pmatrix}
        4 & 9 & 37
    \end{pmatrix}$ be a matrix. In this case we have that $A$ is a complete intersection with unique Markov basis
    \[
    M = \begin{bmatrix}
        9 & -4 & 0 \\
        7 & 1 & -1
    \end{bmatrix}.
    \]
    We define $b = (9, -4, 0)$ and $c = (7, 1, -1)$. Since $c_1 \ge c_2 + c_3$, by Theorem~\ref{thm: dim3 ci}, we have that $M$ is not distance reducing. Furthermore, we have that $2c_2 \ge b_1 + b_2 $ and by an explicit computation we have
    \[
    D(A) = M \cup \pm \begin{bmatrix}
        2 & -5 & 1
    \end{bmatrix}.
    \]
    So, in particular, a unique minimal Markov basis need not be equal to $D(A)$.
\end{example}

\begin{conjecture}
    Fix a matrix $A$. Then $\MD(A) \subseteq \MD^s(A)$.
\end{conjecture}

\subsection{Distance reducing complex}\label{sec: discussion dist red complex} Throughout the paper, we have predominantly focused on the $1$-norm and its induced metric. However, the notion of distance reduction applies to all metrics so we ask how the results of this paper generalise to other metrics. In particular, to other metrics derived from norms. In this section, given a Markov basis $M$, we define the \textit{distance reducing complex}, which is a polyhedral complex whose points correspond to metrics $d$ for which $M$ is distance reducing with respect $d$. We begin by exploring a motivating example.

\medskip \noindent \textbf{Running example.} Fix the matrix 
$A = \begin{pmatrix} 2 & 3 & 4 \end{pmatrix}$, as in Example~\ref{example: A = 2_3_4 continued}. Its Graver basis is:
\[
a_1 := 
\begin{bmatrix}
3 \\ -2 \\ 0
\end{bmatrix}, \ 
a_2 := 
\begin{bmatrix}
2 \\ 0 \\ -1
\end{bmatrix}, \ 
a_3 := 
\begin{bmatrix}
1 \\ -2 \\ 1
\end{bmatrix}, \ 
a_4 := 
\begin{bmatrix}
1 \\ 2 \\ -2
\end{bmatrix}, \ 
a_5 := 
\begin{bmatrix}
0 \\ 4 \\ -3
\end{bmatrix}.
\]
Recall that $A$ has two minimal Markov bases:
\[
M_1 = \{a_1, a_2 \} \quad \text{and} \quad
M_2 = \{a_2, a_3 \}.
\]
And so its indispensable set and universal Markov basis are:
\[
    S(A) = \{a_2\} 
    \quad \text{and} \quad 
    M(A) = \{a_1, a_2, a_3\}.
\]

We now determine the norm-induced metrics $d$ for which $M_1$ is a $d$-reducing Markov basis. First we consider the space of all norm-induced metrics and the possible norms of the elements of the Graver basis.

\medskip
\noindent \textbf{Metric cone}. Given two metrics $d_1$ and $d_2$ on $\ZZ^3$ and a positive real number $\lambda$, we observe that the functions $(\lambda d_1) : (x, y) \mapsto \lambda d_1(x, y)$ and $(d_1 + d_2) : (x, y) \mapsto d_1(x,y) + d_2(x,y)$ define metrics on $\ZZ^3$. So we define the \textit{metric cone} $\MC$ to be the space of norm-induced metrics on $\ZZ^3$. We expect that any norm on the Graver basis extends to a norm on $\ZZ^3$.

\begin{conjecture}
    Suppose that $||\cdot||$ is a norm on a subset $G \subseteq \ZZ^3 \setminus \{0\}$. By this we mean that $||\cdot|| : G \rightarrow \RR_{\ge 0}$ is a function that satisfies:
    \begin{itemize}
        \item $||g|| > 0$ for all $g \in G$,
        \item If $g = \sum_{h \in G} \lambda_h h$ with finitely many nonzero $\lambda_h \in \RR$, then $||g|| \le \sum_{h \in G} |\lambda_h| \cdot ||h||$.
    \end{itemize} Then $||\cdot||$ is the restriction of a norm on $\ZZ^3$.
\end{conjecture}

Suppose $||\cdot||$ is a norm on $\ZZ^3$. Let $n_i := ||a_i|| \in \RR_{>0}$ for each $i \in [5]$. Consider the vector $n = (n_1, n_2, n_3, n_4, n_5) \in \RR^5$. 
If the above conjecture holds, then set of possible vectors $n$ naturally forms a rational polyhedral cone defined by the inequalities that arise from the triangle equality. For example
\[
a_1 = a_2 + a_3 \implies 
n_1 \le n_2 + n_3, \ 
n_2 \le n_1 + n_3, \ 
n_3 \le n_1 + n_2.
\]
The set of inequalities is derived from the set of triangles or, more generally, by the set of linear relations among the given vectors $a_1, \dots, a_5$. The complete set of triangles is given by
\begin{equation} \label{eqn: triangles 1-5}
    123, 12^24, 1^22^35, 13^24, 13^35, 14^35^2,
    234, 23^25, 24^25, 345
\end{equation}
where $\alpha^i\beta^j\gamma^k$ means a triangle whose edges are $i a_\alpha$, $j a_\beta$ and $k a_\gamma$.

As a result, we get a projection of the metric cone $\MC_5 \subseteq \RR_{>0}^5$ whose rays are given by the columns of the following matrix
\[
\begin{bsmallmatrix}
    2 & 1 & 1 & 3 & 0 \\
    1 & 1 & 0 & 2 & 1 \\
    1 & 0 & 1 & 1 & 1 \\
    0 & 1 & 1 & 1 & 2 \\
    1 & 1 & 2 & 0 & 3 
\end{bsmallmatrix}.
\]
The rows of the above matrix correspond to the values of $n_1, n_2, \dots, n_5$ respectively.

\medskip

\noindent \textbf{Distance reducing property.}
We now consider the subset of points $c \in \MC$ such that for any norm $|| \cdot ||$ in the fiber of $c$, the Markov basis $M_1$ is distance reducing with respect to $||\cdot||$. This set has a natural polyhedral-complex structure, which we now describe for this example.

Consider the triangle with sides $a_1, a_2, a_3$, which appears in suitably large fibers of $A$. Let $x$ be the vertex where $a_1$ and $a_3$ meet and $y$ be the vertex where $a_2$ and $a_3$ meet.
By assumption, the basis $M_1$ is distance reducing so either the move from $x$ by $a_1$ reduces the distance or the move from $y$ by $a_2$ reduces the distance. Therefore, either
$n_3 > n_1$ or $n_3 > n_2$.

Next consider the triangle $a_1, 2a_2, a_4$. Starting with the distance $n_4$, we can either reduce it using $a_1$ or $a_2$. If we can reduce it using $a_1$ then we have $a_4 > 2a_2$. Otherwise, if we can reduce it using $a_2$ then we have $a_4 > a_3$.

Continuing in this way, we obtain a set of conditions that must be satisfied by the metric whenever $M_1$ has the distance reducing property. For each triangle: $123$, $12^24$ and $1^22^35$ we reduce $a_3$, $a_4$ and $a_5$ by either moving along $a_1$ or $a_2$. However, for the triangle $1^22^35$, the reduction by move $a_1$ gives rise to a triangle $12^36$ where $a_6 = (3, 2, -3)^T$ does not belong to the Graver basis. The resulting set of conditions are shown in Table~\ref{table: triangle reduction}.

\begin{table}[t]
    \centering
    \begin{tabular}{ccc}
        \toprule
        \multirow{2}{*}{Triangle} & \multicolumn{2}{c}{Reduction conditions} \\
        & Reduce by $a_1$ & Reduce by $a_2$ \\
        \midrule
        $123$       & $n_3 > n_2$   & $n_3 > n_1$ \\
        $12^24$     & $n_4 > 2n_2$  & $n_4 > n_3$ \\
        $1^22^35$   & $n_5 > {\color{blue} n_6}$   & $n_5 > 2n_3$ \\
        \color{blue} $12^36$     & \color{blue} $n_6 > 3n_2$ &  \color{blue} $n_6 > n_4$ \\
        \bottomrule
    \end{tabular}
    \caption{Triangles and their reduction moves. A distance reducing metric satisfies at least one condition from each row.}
    \label{table: triangle reduction}
\end{table}

The vector $a_6$ is assumed to have norm $n_6 := ||a_6||$. The extended metric cone $\MC_6 \subseteq \RR^6_{>0}$ is derived from the triangle inequalities given by previous set of triangles (\ref{eqn: triangles 1-5}) together with:
\begin{equation} \label{eqn: triangles 6}
    12^36, 1^23^36, 14^36^2, 156,
    2^236, 246, 2^356^2,
    34^26, 3^35^26,
    4^356.
\end{equation}
The metric cone $\MC_6 \subseteq \RR^6_{>0}$ is given by the cone over the columns of the matrix:
\[
\begin{bsmallmatrix}
    2 & 1 & 1 & 3 & 0 & 3 \\
    1 & 1 & 0 & 2 & 1 & 1\\
    1 & 0 & 1 & 1 & 1 & 2\\
    0 & 1 & 1 & 1 & 2 & 1 \\
    1 & 1 & 2 & 0 & 3 & 3 \\
    1 & 2 & 1 & 3 & 3 & 0 
\end{bsmallmatrix}.
\]

\begin{remark}
    We may avoid adding $a_6$ into the computation and consider only the projection of the distance-reducing metrics onto $\MC_5$. If we do this, then stay within the cone defined by (\ref{eqn: triangles 1-5}) and we obtain weaker inequalities derived from the triangle inequality. For instance: $n_5 > |3n_2 - n_1|$, which is derived from the reduction by $a_1$, and $n_5 > 2n_3$, which is derived from the reduction by $a_2$.
\end{remark}

\begin{remark}
    The metric cone $\MC_6$ admits a surprising symmetry under a representation of the group $S_3$. Let $\sigma_1 = (1, 2)$ and $\sigma_2 = (2,3)$ be adjacent transpositions that generate $S_3$. The action is given by 
    \[
    \sigma_1 = 
    \begin{bmatrix}
        0 & 1 & 0 \\
        1 & 0 & 0 \\
        0 & 0 & 1 \\
    \end{bmatrix}
    \oplus
    \begin{bmatrix}
        1 & 0 & 0 \\
        0 & 0 & 1 \\
        0 & 1 & 0 \\
    \end{bmatrix}
    \quad \text{and} \quad
    \sigma_2 = 
    \begin{bmatrix}
        1 & 0 & 0 \\
        0 & 0 & 1 \\
        0 & 1 & 0 \\
    \end{bmatrix}
    \oplus
    \begin{bmatrix}
        0 & 0 & 1 \\
        0 & 1 & 0 \\
        1 & 0 & 0 \\
    \end{bmatrix}
    \]
    where $A \oplus B$ denotes the block diagonal matrix 
    $\begin{bmatrix}
        A & 0 \\ 0 & B
    \end{bmatrix}$. Similarly, there is a permutation action that acts on the coordinates of the cone, which suggests that the roles of $a_1$, $a_5$ and $a_6$ are interchangeable, and whenever we permute their roles there is an induced permutation on $a_2$, $a_3$ and $a_4$. One reason why this may happen is because the vectors $a_1, \dots, a_6$ can be partitioned into two triangles $a_1, a_5, a_6$ and $a_2, a_3, a_4$.
\end{remark}

\noindent 
\textbf{Distance reducing complex.}
Every distance reducing metric for $M_1$ satisfies at least one of the inequalities in each row in \Cref{table: triangle reduction}. For each row in the table, we choose one of the inequalities. For each collection of such choices, we intersect the corresponding open halfspaces and take their intersection with the metric cone $\MC_6$. This gives a collection of relatively open cones that define the \textit{distance reducing complex}. For example, if we choose the inequalities
\[
I_1 = \{n_3 > n_2, \ n_4 > 2n_2, \ n_5 > n_6, \ n_6 > 3n_2 \},
\]
which is obtained by always reducing by $a_1$,
then the intersection of the cone defined by $I_1$ and $\MC_6$ is the cone $\cone(A_1)$ over the columns of the matrix
\[
A_1 = 
\begin{bsmallmatrix}
    1 & 0 & 2 & 3 & 4 & 6 \\
    0 & 1 & 1 & 1 & 1 & 1 \\
    1 & 1 & 1 & 2 & 3 & 5 \\
    1 & 2 & 2 & 2 & 2 & 4 \\
    2 & 3 & 3 & 3 & 5 & 9 \\
    1 & 3 & 3 & 3 & 3 & 3
\end{bsmallmatrix}.
\]
On the other hand, if we choose the inequalities
\[
I_2 = \{n_3 > n_2, \ n_4 > 2n_2, \ n_5 > n_6, \ n_6 > n_4 \},
\]
then we obtain the cone $\cone(A_2)$ over the columns of the matrix
\[
A_2 = 
\begin{bsmallmatrix}
    1 & 2 & 3 & 4 & 0 & 2 & 3 & 4 & 7 \\
    0 & 1 & 1 & 1 & 1 & 1 & 1 & 1 & 2 \\
    1 & 2 & 2 & 3 & 1 & 1 & 2 & 3 & 5 \\
    1 & 2 & 2 & 2 & 2 & 2 & 2 & 2 & 4 \\
    2 & 4 & 2 & 5 & 3 & 3 & 3 & 5 & 8 \\
    1 & 2 & 2 & 2 & 3 & 3 & 3 & 3 & 4
\end{bsmallmatrix}.
\]
Note that, despite choosing different inequalities, the interiors of the cones $\cone(A_1)$ and $\cone(A_2)$ intersect. Their intersection $\cone(A_{12})$ is the cone over the columns of the matrix
\[
A_{12} = 
\begin{bsmallmatrix}
    1 & 0 & 2 & 3 & 4 & 3 & 4 & 5 & 9 \\
    0 & 1 & 1 & 1 & 1 & 1 & 1 & 1 & 2 \\
    1 & 1 & 1 & 2 & 3 & 3 & 3 & 4 & 7 \\
    1 & 2 & 2 & 2 & 2 & 3 & 3 & 3 & 6 \\
    2 & 3 & 3 & 3 & 5 & 6 & 3 & 7 & 12 \\
    1 & 3 & 3 & 3 & 3 & 3 & 3 & 3 & 6
\end{bsmallmatrix}.
\]
This cone corresponds to the choice of inequalities $I_1 \cup I_2$.

\medskip
\noindent \textbf{Reduction complex.}
The natural generalisation of the complex in the above example is constructed as follows.

\begin{definition}[$B$-reduction closed]
     Let $B$ be a Markov basis for $A \in \ZZ^{d \times n}$. We say that a set $S \subseteq \ker(A)$ is \textit{closed under $B$-reductions} if $B \subseteq S$ and if for any linear relation $\alpha_b b + \sum_{b' \in B \backslash \{b\}} \alpha_{b'} b' = \alpha_s s $ with $b \in B$, $\alpha_b \ge 1$, $s \in S$, $\alpha_s$ nonzero and $\alpha_x\in \ZZ$ for all $x$, there exists $s' \in S$ and $\alpha_{s'} \in \ZZ$ nonzero such that the linear relation $(\alpha_b - 1)b +\sum_{b' \in B \backslash \{b\}} \alpha_{b'} b' = \alpha_{s'} s'$ holds. We call the second linear relation the \textit{reduction} of the first linear relation with respect to $b$. We say that the reduction is \textit{unique} with respect to $b$, if $s'$ above is unique.
\end{definition}

If a finite set $S \supseteq B$ is not closed under $B$-reductions, then it can be made closed by adding finitely many elements of $\ker(A)$. If we assume that only primitive elements may be added, then this closure is unique and defines a closure operation is the usual sense. If $S$ contains only primitive elements, then all reductions are unique.

\begin{definition}[Reduction inequalities]
    Let $B$ be a Markov basis and fix a finite set $S \subseteq \ker(A)$ closed under $B$-reductions. For each linear relation $\alpha_b b + \sum_{b' \in B \backslash \{b\}} \alpha_{b'} b' = \alpha_s s $ with $b \in B$, $\alpha_b \ge 1$, $s \in S$, $\alpha_s$ nonzero and $\alpha_x\in \ZZ$ for all $x$ that admits a reduction to $(\alpha_b - 1)b +\sum_{b' \in B \backslash \{b\}} \alpha_{b'} b' = \alpha_{s'} s'$ with respect to $b$, we define the \textit{reduction inequality} $|\alpha_s| n_s > |\alpha_{s'}| n_{s'}$, which defines a half-space in $\RR^{S}$ whose coordinates are denoted $n_i$ for each $i \in S$.
\end{definition}

\begin{definition}[Distance reducing complex]
    Let $B$ be a Markov basis and $S \subseteq \ker(A)$ be a finite set closed under $B$-reductions. Let $\ML = \{ \sum_{b \in B} \alpha_b b = \alpha_s s \colon \alpha_s \neq 0\}$ be the set linear relations that admit reductions. We identify a relation  $\sum_{b \in B} \alpha_b b = \alpha_s s$ with its negative $\sum_{b \in B} (-\alpha_b) b = -\alpha_s s$, so we assume that any individual coefficient $\alpha_b$ is non-negative. For each relation $L \in \ML$, we obtain the set of reduction inequalities $I_L = \{|\alpha_s| n_s > |\alpha_{s'}| n_{s'}\}$ with one inequality for each $b \in B$ such that $\alpha_b \neq 0$. The \textit{distance reducing complex $\Delta(B, S) \subseteq \RR^S$} is the union $\Delta(B, S) = \bigcup_{T} \MC_T$ where $T$ runs over all \textit{transversals} of $\{I_L \colon L \in \ML\}$ and $\MC_T = \bigcap_{t \in T} H_t$ is the intersection of the half-spaces $H_t \subseteq \RR^S$ defined by the inequality $t \in T$.
\end{definition}

By taking further intersections, and recording the inequalities, it is possible to imbue a natural polyhedral complex structure to $\Delta(B, S)$. Explicitly, the maximal cones are indexed by super-sets of transversals of $\{I_L\}$.

\begin{question}
    For which metrics is the Markov basis \textit{minimally distance reducing}? 
    
\end{question}

\begin{remark}
    For other matrices, it may be possible for the reduction conditions to be satisfied but the Markov basis to not be distance reducing.
\end{remark}

\medskip \noindent
\textbf{General metric cone.}
Let $S \subseteq \ker(A)$ be a non-empty collection of vectors. This set realises a matroid $\MM(S)$ with circuits $\MC(S)$. We define the metric cone in terms of the circuits of this matroid. For example, if $123 \in \MC(S)$ is a circuit then it defines a triangle so we get three inequalities:
\[
|c_1|n_1 \le |c_2|n_2 + |c_3|n_3, \quad 
|c_2|n_2 \le |c_1|n_1 + |c_3|n_3, \quad
|c_3|n_3 \le |c_1|n_1 + |c_2|n_2.
\]
On the other hand if we have a larger circuit, like $1 \dots m$, then we get weaker inequalities:
\[
|c_i|n_i \le |c_1|n_1 + \dots + \widehat{|c_i|n_i} + \dots + |c_m|n_m \text{ for each } i \in [m]
\]
The reason why these inequalities are `weaker' is easily seen when $\MM(S)$ is a graphic matroid of a chordal graph $G$. In this case, if $C \in \MC(S)$ is a circuit, then it is cycle in $G$. Since $G$ is choral, the cycle admits a triangulation. The triangle inequalities arising from the triangles imply the inequalities derived from $C$.

So we obtain a metric cone $\MC_S \subseteq \RR^S$ defined by all the circuits of $\MM(S)$.

\begin{question}
    We ask the following questions about the distance reducing complex. 
\begin{itemize}
    \item When do two sets of vectors $S_1$ and $S_2$ define the same metric cone $\MC_{S_1} = \MC_{S_2}$?

    \item What is the structure of the distance reducing complex for monomial curve?

    \item How do the results for the $1$-norm, such as Theorems~\ref{thm: characterisation dist red first kind} and \ref{thm: dist red ci implies first kind}, generalise to other metrics?
\end{itemize}
\end{question}

{\bf Acknowledgments.} D. Kosta gratefully acknowledges funding from the Royal Society Dorothy Hodgkin Research Fellowship DHF$\backslash$R1$\backslash$201246. We thank the associated Royal Society Enhancement grant RF$\backslash$ERE$\backslash$210256  that supported  O. Clarke's postdoctoral research.

\bibliographystyle{plain}
\bibliography{references}

\end{document}